\documentclass[11pt]{article} 

\usepackage{empheq}
\usepackage{amsmath,amsbsy,amssymb,amsfonts}
\usepackage[applemac]{inputenc}
\usepackage{pifont} 
\usepackage{shadethm}
\usepackage{subfigure,graphicx}
\usepackage{lineno}
\usepackage{dsfont}

\def\a {{\boldsymbol a}}

\def\d {{\boldsymbol d}}
\def\f {{\boldsymbol f}}

\def\u {{\boldsymbol u}}
\def\n {{\boldsymbol n}}
\def\x {{\boldsymbol x}}
\def\v {{\boldsymbol v}}
\def\w {{\boldsymbol w}}

\def\utilde{{\widetilde{\boldsymbol u}}}
\def\utecho{{\widehat{\boldsymbol u}}}
\def\Om {{\Omega}}

\def\D{{\boldsymbol D}}

\def\H{{\boldsymbol H}}
\def\V{{\boldsymbol V}}
\def\L{{\boldsymbol L}}
\def\W{{\boldsymbol W}}
\def\X{{\boldsymbol  X}}
\def\Y{{\boldsymbol  Y}}

\def\Emat{{\boldsymbol{\mathsf{E}}}}
\def\Mmat{{\boldsymbol{\mathsf{M}}}}
\def\Cmat{{\boldsymbol{\mathsf{C}}}}

\def\Jmat{{\boldsymbol{\mathsf{J}}}}
\def\Lmat{{\boldsymbol{\mathsf{L}}}}

\def\Wvec{\mathsf{W}}
\def\Dvec{\mathsf{D}}
\def\Uvec{\mathsf{U}}
\def\Pvec{\mathsf{P}}

\def\Fvec{\mathsf{F}}

\def\R {{\rm I}\hskip -0.85mm{\rm R}}

\newtheorem{theorem}{Theorem}
\newtheorem{lemma}[theorem]{Lemma}
\newtheorem{proposition}[theorem]{Proposition}
\newtheorem{corollary}[theorem]{Corollary}
\newenvironment{proof}{{\it  Proof}.}{\hfill $\square$\newline}

\title{A time-splitting finite-element approximation for the Ericksen-Leslie equations}

\author{R.C. Cabrales\thanks{Grupo de Matem\'aticas Aplicadas,
Universidad del B\'\i o-B\'\i o, Casilla 447, Chill\'an, Chile.
E-mail: {\tt roberto.cabrales@gmail.com}. Partially supported under grants
GI 121909/C Universidad del B\'\i o-B\'\i o, Chile and 
Ministerio de Econom\'ia y Competitividad under grant MTM2012-32325, Spain.} 
\and F. Guill\'en-Gonz\'alez\thanks{Dpto.~E.D.A.N., Universidad 
de Sevilla Sevilla, Aptdo.~1160, 41080 Sevilla, Spain. 
E-mail: {\tt guillen@us.es}. Partially supported by  Ministerio de Econom\'ia y Competitividad under grant MTM2012-32325, Spain. }
\and J. V. Guti\'errez-Santacreu\thanks{ Dpto. de 
Matem\'atica Aplicada I, Universidad de Sevilla, E.
T. S. I. Inform\'atica. Avda. Reina Mercedes, s/n. 41012 Sevilla,
Spain. {\tt juanvi@us.es}. Partially supported by  Ministerio de Econom\'ia y Competitividad and Ministerio de Educaci\'on under grants MTM2012-32325 and JC2011-418, Spain.} }

\begin{document}

\maketitle

\begin{abstract}      
In this paper we propose a time-splitting finite-element scheme for approximating solutions of the Ericksen-Leslie equations governing the flow of nematic liquid crystals. These equations are to be solved for a velocity vector field and a scalar pressure   as well as a director vector field representing the direction along which the molecules of the liquid crystal are oriented. 

The algorithm is designed at two levels. First, at the variational level, the velocity,  pressure and director  are computed separately, but the director field has to be computed together with an auxiliary variable in order to deduce a priori energy estimates. Second, at the algebraic level, one can avoid computing such an auxiliary variable if this is approximated by a piecewise constant finite-element space.  Therefore, these two steps give rise to a numerical algorithm that computes separately only the primary variables. Moreover, we will use a pressure stabilization technique that allows an stable equal-order interpolation for the velocity and the pressure. Finally, some numerical simulations are performed in order to show the robustness and  efficiency of the proposed numerical scheme and its accuracy. 
\end{abstract}

\noindent{\bf Mathematics Subject Classification:}Nematic liquid crystal; Finite elements; Projection method; Time-splitting method.


%

\section{Introduction}

There has been a great interest in the finite-element numerical approximation of liquid crystal flows in recent years. The reason for this is that liquid crystals are not easy to be studied from experimental observations due to the effect of boundary conditions of the confining geometries. Thus, numerical simulations allow a clear insight into the behavior of liquid crystals and the understanding of their underlying physical properties. For instance, numerical simulations contribute to improve the design of practical devices.

Liquid crystals are materials that show intermediate transitions between a solid and a liquid called {\it mesophases}. It means that  liquid crystals combine properties of both an isotropic liquid and  a crystalline solid. These mesophases are due, in part, to the  fact that liquid crystals are made of macromolecules of similar size, which are commonly represented like rods or plates. It is also known that the shape of the molecules play an important role in such mesophases.  Moreover, liquid crystals depend on the temperature (thermotropic) and/or the concentration of a solute in a solvent (lyotropic) so that they can change from liquid to solid by means of varying the temperature and/or the concentration.  

The mathematical theory describes liquid crystals attending to the different degrees of positional or orientational ordering of their molecules. Thus, the positional order alludes to the position of the molecules while the orientation order referred to the fact that the molecules to tend to be locally aligned towards certain preferred direction. Such a direction is described by a unit vector along the molecule if rod-shaped or perpendicular to the molecule if plate-shaped measuring the mean values of alignments. 

The simplest phase of liquid crystals is called {\it nematic} which 
possesses an orientational ordering but not positional. That is, the 
molecules flow freely as in a disordered isotropic liquid phase while 
tend to be orientated along a direction which can be manipulated 
with mechanical (boundary conditions), magnetic or electric forces.

The simplest phenomenological description of spatial 
configurations in nematic liquid crystals is the Oseen-Frank 
theory \cite{Oseen, Frank}. This approach consists in modeling 
equilibrium states as minima of a free-energy functional which is set 
up through symmetry and invariance principles, to capture some 
properties observed from experiments. Thus, the Oseen-Frank free 
energy is considered as a functional of the director vector $\d$. 
In its most basic form, the free energy functional is given by
\begin{equation*}
E(\d)=\int_{\Omega}\left\{K_1|\nabla\cdot\nabla
\d|^2+K_2(\d\cdot (\nabla\times \d))^2
+K_3(\d\times(\nabla\times\d))^2\right\},
\end{equation*}
where $K_1$, $K_2$, and $K_3$ are the \textit{splay},
\textit{twist}, and \textit{bend} elastic constants, respectively. 
Note that when these constants are equal, the Dirichlet energy becomes
$$
E(\d)=K\int_{\Omega}|\nabla \d|^2.
$$
Upon minimizing this energy subject to the sphere constraint
$|\d|=1$, the following optimality system appears
\begin{equation*}
-\Delta\d-|\nabla\d|^2\d=\boldsymbol{0}\quad \mbox{ in }\quad\Omega.
\end{equation*}

The limitation of the Oseen-Frank theory relies on the fact that 
it can only explain point defects in liquid crystal materials but 
not the more complicated line and surface defects that are also 
observed experimentally. The defect points or singularities in liquid 
crystals are regions where the anisotropic properties of molecules 
are broken. That is, the liquid crystal behaves as an isotropic fluid. 
Therefore, the director field cannot be defined. Mathematically,
they are modeled by $|\d|=0$. One way of inducing defect points 
is with the help of the boundaries conditions. 

The motion of defect points in liquid crystals can be studied via 
the long-time behavior of the harmonic map flow for which it is also
interesting to incorporate the influence of the velocity. 
On the contrary, in many situations, the anisotropic local orientation 
of the director field influences the stress tensors that govern the 
fluid velocity. The hydrodynamic theory of nematic liquid crystals 
was established by Ericksen \cite{Ericksen1, Ericksen2} and Leslie
\cite{Leslie1, Leslie2}. The fundamental system consists of a set 
of fully coupled, macroscopic equations, that contains the 
Oseen-Frank elastic theory governing the steady state, 
equilibrium solutions.

The remaining part of this paper is organized as follows. Section 2 
starts by establishing some notation used throughout this paper. 
Then we follow with the differential formulation of the Ericksen-Leslie 
and the Ginzburg-Landau equations. To end the section, we sum up the main
 contributions on the finite-element approximation of the Ginzburg-Landau 
 equations. In Section 3 we give some short-hand notation for 
 finite-element spaces in order to be able to define the projection 
 time-stepping algorithm and give a brief introduction to some key 
 ideas leading to the proposed method. Next, in Section 4, we prove a 
 priori estimates for the algorithm. Section~5 is devoted to some 
 implementation improvements. Finally, we validate the numerical 
 scheme with some simulations.         

\section {Statement of the problem}

Let $\Omega\subset \R^M, M=2,3$ be any bounded open 
set with boundary $\partial\Omega$. For $1\le p \le \infty$, $L^p(\Omega)$ denote the space of $p$th-power 
integrable real-valued functions defined on $\Omega$  for the Lebesgue 
measure. This space is a Banach space endowed with the norm
$\|v\|_{L^p(\Omega)}=(\int_{\Omega}|v(\x)|^p\,{\rm d}\x)^{1/p}$ for $1\le p <\infty$ or
$\|v\|_{L^\infty(\Omega)}={\rm ess}\sup_{\x\in \Omega}|v(\x)|$ for $p=\infty$. In particular, $L^2(\Omega)$ is a Hilbert space with 
the inner product
$$
\left(u,v\right)=\int_{\Omega}u(\x)v(\x){\rm d}\x,
$$
and its norm is simply denoted by $\|\cdot\|$.
For $m$ a non-negative integer, we define the classical Sobolev spaces as 
$$
H^{m}(\Omega) = \{v \in
L^p(\Omega)\,;\, \partial^k v \in L^p(\Omega)\ \forall ~ |k|\le m\},
$$ 
associated to the norm 
$$
\|v\|_{H^{m}(\Omega)} =\left[\sum_{0\le |k| \le m} \|\partial^k
v\|^2\right]^{\frac{1}{2}}\,,
$$ 
where $k = (k_1,\ldots,k_M)\in{\mathds{N}^M}$ is a
multi-index and $|k| = \sum_{i=1}^M k_i$, which is 
a Hilbert space with the obvious inner
product. We will use boldfaced letters 
for spaces of vector functions and their elements, e.g. $\L^2(\Omega)$ in 
place  of $L^2(\Omega)^M$.

Let $\mathcal{D}(\Omega)$ be the space of infinitely times
differentiable functions with compact support on $\Omega$. 
The closure of ${\cal D}(\Omega)$ in
$H^{m}(\Omega)$ is denoted by $H^{m}_0(\Omega)$. 
We will also make use of the following space of vector fields:
$$
\boldsymbol{\mathcal{V}}=\{\v\in \boldsymbol{\mathcal{D}}(\Omega): \nabla\cdot\v=0 \mbox{ in } \Omega \}. 
$$
We denote by $\H$ and $\V$, the closures of $\boldsymbol{\mathcal{V}}$, in the $\L^2(\Omega)$- and $\H^1(\Omega)$-norm, respectively,   
which are characterized by (see \cite{Temam})
\begin{eqnarray*}
\H&=& \{ \u \in \L^2(\Omega) : \nabla\cdot\u =0 \mbox{ in } 
\Omega, \u\cdot\boldsymbol{n} = 0 \hbox{ on }
\partial\Omega \},\\
\V&=& \{\u \in \H^1(\Omega) : \nabla\cdot\u =0 \mbox{ in } 
\Omega, \u = \boldsymbol{0} \hbox{ on } \partial\Omega \},
\end{eqnarray*}
where $\n$ is the outward normal to $\Omega$ on $\partial \Omega$. This characterization is valid for $\Omega$ being Lipschitzian.   
Finally, we consider 
$$
L^2_0(\Omega)= \{ p \in L^2(\Omega) : \ \int_\Omega
p(\x)\, d\x =0 \}.
$$

\subsection {The Ericksen-Leslie problem}

Let $T>0$ be a  fixed time. We will use the notation $Q=\Omega\times(0,T)$ and 
$\Sigma=\partial\Omega\times(0,T)$. The Ericksen-Leslie equations are written as
\begin{subequations}\label{Ericksen-Leslie-Problem}
\begin{empheq}[left=\empheqlbrace]{align}
\partial_t \d+  (\u\cdot \nabla) \d - \gamma\Delta\d  
- \gamma |\nabla \d|^2 \d &=\mathbf {0}\quad\mbox{ in $Q$},\label{ELP1}\\
|\d|&= 1 \quad\mbox{ in $Q$},\label{ELP2}\\
\partial_t\u + (\u\cdot \nabla) \u - \nu \Delta \u+
\nabla p + \lambda\nabla\cdot ((\nabla \d )^T \nabla \d)
         & = \mathbf {0}\quad\mbox{ in $Q$,}\label{ELP3}\\
        {\nabla \cdot} \,\u & =  0\quad\mbox{ in $Q$},\label{ELP4}
\end{empheq}
\end{subequations}
where $\u:\overline Q \to \R^M$ is the fluid velocity, $p: \overline Q \to \R$ is 
the fluid pressure, and $\d: \overline Q \to \R^M$ is the orientation 
of the molecules. The parameter  $\nu>0$ is a constant depending 
on the fluid viscosity, $\lambda>0$ is an elasticity constant, and 
$\gamma>0$ is a relaxation time constant. The operators involve in system 
\eqref{Ericksen-Leslie-Problem} are described as follows. 
The Laplacian operator $\Delta \u=\sum_{i=1}^M\partial_{ii}\u $, 
the convective operator $(\u\cdot\nabla)\w=\sum_{i=1}^M u_i\partial_i\w$, 
and the divergence operator $\nabla\cdot\u=\sum_{i=1}^M \partial_i u_i$. 
 Moreover, $(\nabla \d)^T$ denotes the transposed matrix of $\nabla
  \d=(\partial_j d_i)_{i,j}$ and $|\d|=|\d(\x,t)|$ is the Euclidean 
 norm in $\R ^M$.

The system \eqref{Ericksen-Leslie-Problem} provides a 
phenomenological description for the hydrodynamics of 
nematic liquid crystals from the macroscopic 
point of view. It was reduced to essentials by Lin 
\cite{Lin} from the fundamental  set of fully, coupled, 
macroscopic equations  derived by Ericksen 
\cite{Ericksen1, Ericksen2} and Leslie 
 \cite{Leslie1, Leslie2}, that contains the Oseen-Frank 
 elastic energy governing the steady state, equilibrium 
 solutions of nematic liquid crystals.

Equation $\eqref{ELP1}$ is the equation for the conservation of the 
angular momentum; in particular, it is a convective 
harmonic heat map flow equation into spheres together with equation 
$\eqref{ELP2}$. This latter indicates that $\d$ is not a state 
variable, it only describes the orientation of the 
nematic liquid crystal molecules. Equations 
$\eqref{ELP3}$ and  $\eqref{ELP4}$ 
are the Navier-Stokes equations related to  the 
conservation of the linear momentum. 
The molecules add (elastic) stress to the fluid via the term 
$\lambda\nabla\cdot ((\nabla \d )^T \nabla \d)$ and the fluid carries 
the molecules via the term $(\u\cdot\nabla)\d$.  
  
To these equations we will add homogeneous Dirichlet conditions for the velocity field  and homogeneous Neumann  boundary conditions  for the director field:
\begin{equation}\label{boundary-conditions}
    \u(\x,t)=\boldsymbol{0}, \quad \partial_{\n}\d(\x,t)= \boldsymbol{0}
    \quad\mbox{ for $(\x,t)\in\Sigma$,}
\end{equation}
and the initial conditions
\begin{equation}\label{initial-conditions}
    \d(\textit{\textbf{x}},0)=\d_0(\textit{\textbf{x}}),\quad \quad
    \u(\textit{\textbf{x}},0)=\u_0(\textit{\textbf{x}})\quad\mbox{ for
    $\x\in\Om$.}
\end{equation}
Here  $\u_0:\Omega\to\R^M$, with $\u_0\in \H$, and
$\d_0:\Omega\to\R^M$, with $\d\in \H^1(\Omega)$ satisfying 
$|\d|=1$ in $\Omega$, are given functions.
One can prove the following energy law for system 
\eqref{Ericksen-Leslie-Problem}:
\begin{equation}\label{Ericksen-Leslie-Energy}
\frac{d}{dt}\left( \frac{1}{2}\|\u\|^2 
+ \frac{\lambda}{2} \|\nabla\d\|^2\right) +
\nu\|\nabla\u\|^2 + \lambda\gamma \|\Delta\d+|\nabla\d|^2\d\|^2 = 0,
\end{equation}
but it requires that $\d$ must have the unit length, i.e., $|\d|=1$ 
almost everywhere in $Q$. It makes system \eqref{Ericksen-Leslie-Problem}
difficult to manage from the numerical point of view since the 
satisfaction of the sphere constraint at the nodes does not imply 
at any other points via interpolation. For this reason, two approaches 
have been considered for dealing with it: a penalty method and a 
saddle-point method. These techniques provide numerical schemes with 
an associated energy law without the need of satisfying the sphere 
constraint for $\d$. The penalty method has intensively studied over 
the saddle-point strategy since this latter is more challenging to 
perform the numerical analysis rigorously. The difficulty lies in 
proving an inf-sup condition for the Lagrangian multiplier related 
to the sphere constraint. In order for such an inf-sup condition 
\cite{Hu-Tai-Winther} to hold, a stronger regularity than the one 
provided by \eqref{Ericksen-Leslie-Energy} is needed; therefore 
establishing an inf-sup condition under the regularity stemmed from 
\eqref{Ericksen-Leslie-Problem} is still an interesting, open problem.   

\subsection{The Ginzburg-Landau problem}
 
The penalization argument is typically based on the Ginzburg-Landau 
penalty function \cite{Lin}. Thus, system \eqref{Ericksen-Leslie-Problem} 
in its penalty version reads as:         
\begin{subequations}\label{Ginzburg-Landau-Problem}
\begin{empheq}[left=\empheqlbrace]{align}
\partial_t \d+  \u\cdot \nabla \d + \gamma ( \f_{\varepsilon}(\d)  -
        \Delta\d ) & = \mathbf{0}\quad\mbox{ in  $Q$,}\label{GLP1}\\
        \partial_t\u + \u\cdot \nabla \u - \nu \Delta \u+
        \nabla p + \lambda\nabla\cdot ((\nabla \d)^T \nabla \d)
        & = \boldsymbol{0}\quad\mbox{ in $ Q$,}\label{GLP2}\\
        \nabla\cdot \,\u & =  0\quad\mbox{ in $ Q$,}\label{GLP3}
\end{empheq}
\end{subequations}
where
\begin{equation}\label{Pen_func}
\f_{\varepsilon} (\d)= \frac{1}{\varepsilon^{2}}\left(|\d|^2-1\right)\d,
\end{equation}
is the penalty function related  to  the constraint $|\d|=1$, and
$\varepsilon>0$ is the penalty parameter. It is important to observe
that $\f_{\varepsilon}$ is the gradient of the scalar potential
function
$$
F_{\varepsilon} (\d)=\frac{1}{4\varepsilon^2}(|\d|^2-1)^2,
$$
that is,  $\f_{\varepsilon} (\d) = \nabla_\d F_{\varepsilon}(\d)$
for all $\d\in \R^M$.  The virtue of system \eqref{Ginzburg-Landau-Problem} is that its solutions satisfy an energy law without assuming any 
restriction on $\d$  as was mentioned above. We give here a sketch of 
the proof of the energy estimate obtained in \cite{Becker-Feng-Prohl} 
based on that of \cite{Lin-Liu}  in order to have a clear picture of 
how our numerical scheme is designed. First, note that  
$$
\lambda\nabla\cdot((\nabla\d)^T\nabla\d)=\lambda\nabla
\left(\frac{1}{2}|\nabla\d
|^2+F_\varepsilon(\d)\right)-\lambda(\nabla\d)^T(-\Delta
\d+\f_\varepsilon(\d)),
$$
and 
$$ 
(\u\cdot\nabla)\d \cdot (-\Delta\d+\f_\varepsilon(\d))=
(\nabla\d)^T(-\Delta\d-\f_\varepsilon(\d))\cdot\u.
$$ 
Next, multiplying equations $\eqref{GLP1}$ and $\eqref{GLP2}$ by 
$-\Delta\d+\f_\varepsilon(\d)$ and $\u$, respectively,  and 
integrating over $\Omega$, we obtain, after some integrations by parts: 
\begin{equation}\label{Ginzburg-Landau-Energy}
\frac{d}{dt}\left(\frac{1}{2}\|\u\|^2+\frac{\lambda}{2}\|\nabla\d
\|^2 + \lambda \int_{\Omega}F_{\varepsilon} (\d)\right)
+\nu\|\nabla\u\|^2+\lambda\gamma \|-
\Delta\d+\f_{\varepsilon}(\d)\|^2= 0.
\end{equation}
Following the ideas presented in Sections 3 and 4, one could design a time-splitting scheme for the penalization function (\ref{Pen_func}) which has a priori energy estimates; even though a slightly more complicated arguments must be given.  However, this scheme would lead to a stronger constraint for the space, time and penalty parameters than what we will obtain if we use the following truncated potential function \cite{Guillen-Gutierrez}:
\begin{equation}\label{Truncated-Pontential-fun}
\widetilde F_\varepsilon(\d)=
\begin{cases}
\displaystyle
\frac{1}{4\varepsilon^2}(|\d|^2-1)^2, &\mbox{ if } |\d|\le 1,\\
& \\
\displaystyle
\frac{1}{\varepsilon^2}(|\d|-1)^2, &\mbox{ if } |\d|> 1,
\end{cases}
\end{equation}
for which
$$
\nabla_\d \widetilde F(\d)= 
\widetilde\f_\varepsilon(\d)=
\begin{cases}
\displaystyle
\frac{1}{\varepsilon^2}(|\d|^2-1) \d, &\mbox{ if } |\d|\le 1,\\
& \\
\displaystyle
\frac{2}{\varepsilon^2} (|\d|-1)\frac{\d}{|\d|}, &\mbox{ if } |\d|> 1.
\end{cases}
$$
Therefore, system  \eqref{Ginzburg-Landau-Problem} reminds as
\begin{subequations}\label{Ginzburg-Landau-Problem-Truncation}
\begin{empheq}[left=\empheqlbrace]{align}
\partial_t \d+  \u\cdot \nabla \d+ 
\gamma (\widetilde \f_{\varepsilon}(\d)  -
        \Delta\d ) & = \mathbf{ 0}\quad\mbox{ in  $Q$,}\label{GLPT1}\\
\partial_t\u + \u\cdot \nabla \u - \nu \Delta \u+
        \nabla p + \lambda\nabla\cdot ((\nabla \d)^T \nabla \d)
        & = \mathbf{ 0}\quad\mbox{ in $ Q$,}\label{GLPT2}\\
        \nabla\cdot \,\u & =  0\quad\mbox{ in $ Q$.}\label{GLPT3}
    \end{empheq}
\end{subequations}
The first question to be set out is whether system 
\eqref{Ginzburg-Landau-Problem-Truncation} has a priori estimates 
equivalent to \eqref{Ericksen-Leslie-Energy}. The energy for 
system \eqref{Ginzburg-Landau-Problem-Truncation} is followed 
in the same way we did to obtain the energy law 
\eqref{Ginzburg-Landau-Energy}:
\begin{equation*}
\frac{d}{dt}\bigg(\frac{1}{2}\|\u\|^2+\frac{\lambda}{2}\|\nabla\d
\|^2 + \lambda \int_{\Omega}\widetilde F_{\varepsilon} (\d)\bigg)
+\nu\|\nabla\u\|^2_{\L^2(\Omega)}+\lambda\gamma \|-
\Delta\d+\widetilde\f_{\varepsilon}(\d)\|^2= 0.
\end{equation*}
The second question to be asked is the relationship between 
systems \eqref{Ginzburg-Landau-Problem} and 
\eqref{Ginzburg-Landau-Problem-Truncation}. The following lemma 
clarifies this situation. A detailed proof can be found in \cite{Badia-Guillen-Gutierrez-Rev}.

\begin{lemma} If $|\d_0|\le 1$ a.e. in $\Omega$ holds, then 
$|\d|\le 1$ a.e. in $\Omega$ for both systems 
\eqref{Ginzburg-Landau-Problem} and  
\eqref{Ginzburg-Landau-Problem-Truncation}. Then these systems
are equivalent.
\end{lemma}

The Ginzburg-Landau equations \eqref{Ginzburg-Landau-Problem} 
or \eqref{Ginzburg-Landau-Problem-Truncation}  can be viewed as 
being a regularization of the Ericken-Leslie equations \eqref{Ericksen-Leslie-Problem} since one 
can prove the extra regularity estimate \cite{Becker-Feng-Prohl}:
$$
\int_{0}^T\|\Delta\d(t)\|^2\, dt\le C \varepsilon^{-2}.
$$    
Obviously, such an estimate has no meaning as the penalization 
parameter $\varepsilon$ goes to zero. 

\subsection{ Known results}

We discuss briefly the previous numerical schemes on the Ginzburg-Landau problem.  The first two numerical schemes for problem \eqref{Ginzburg-Landau-Problem} were the works of Liu and Walkington 
\cite{Liu-Walkington-2000, Liu-Walkington-2002}. The former used an implicit Euler time-stepping scheme together with LBB-stable finite elements for velocity and pressure and Hermite bicubic $C^1$ finite elements for director. Nevertheless, the numerical resolution was limited by the number of degrees of freedom per rectangle together with the fact that the performance is not an easy task, due to the set of finite element basis functions that connected derivatives up to second order. The latter used the same time discretization but now took advantage of using the auxiliary variable $\w=\nabla \d$ in order to rule out the complexity of using $C^1$-finite element; even though the number of new unknowns made the algorithm inefficient for large scale simulations due to the amount of computational work involved in the process of resolution. Afterwards came the work of Lin and Liu \cite{Lin-Liu-06} who utilized a semi-explicit Euler  time-stepping algorithm, where the stress tensor $\nabla\cdot ((\nabla\d)^T\nabla\d)$ was explicitly discretized, separating the computation of the velocity and pressure from that of the director. Girault and Guill\'en-Gonz\'alez\cite{Girault-Guillen} introduced the auxiliary variable $\w=-\Delta \d$ in order to design a semi-explicit numerical scheme where the Ginzburg-Landau function was explicitly discretized. It is clear that the use of the Laplacian operator in place of the gradient operator reduced considerable the number of global unknowns. One common features of all these numerical schemes described above is that no discrete energy law equivalent to \eqref{Ginzburg-Landau-Energy} was proven independent of $\varepsilon$.

The only numerical scheme \cite{Becker-Feng-Prohl} that preserved a discrete version of \eqref{Ginzburg-Landau-Energy} made use of the auxiliary variable $\w=-\Delta\d+\f_\varepsilon(\d)$ along with an explicitly time discretization of the linear part of the Ginzburg-Landau function. Following the same ideas as in \cite{Becker-Feng-Prohl}, a semi-explicit Euler time-stepping scheme have been considered in \cite{Guillen-Gutierrez}, but this time taking the time discretization of the truncated Ginzburg-Landau function \eqref{Truncated-Pontential-fun} to be fully explicit. The algorithm presented in this paper is based on that  in \cite{Guillen-Gutierrez} being crucial the fully explicit time integration of the Ginzburg-Landau function. 
  
Recently, in \cite{Badia-Guillen-Gutierrez-SP}, a saddle-point strategy has been suggested for both the Ericksen-Leslie and the Ginzburg-Laundau equations arising numerical algorithms which maintain an energy equality comparable to \eqref{Ginzburg-Landau-Energy}. 
The reader is referred to \cite{Badia-Guillen-Gutierrez-Rev} for a survey of numerical methods on the Ginzburg-Landau approximation. 

\subsection{The main contribution of this paper}
An important observation concerning numerical schemes which embody energy estimates from the original problem is that the time integration couples all the unknowns; therefore, the computational work required to solve a time step makes them intensively expensive for approximating the Ginzburg-Landau equations. Therefore, the difficulty in designing an efficient numerical approximation for the Ginzburg-Landau equations lies in choosing a time discretization that, as well as providing energy estimates independent of the penalization parameter, decouples all the variables being computed. But it is also highly desirable to compute only the primary variables steering clear of any auxiliary variable. Observe that the Ginzburg-Landau equations consist of the Navier-Stokes equations with an extra viscous stress tensor to govern the velocity and the pressure, and a convective harmonic map heat flow equation to govern the dynamics of the director field. 

Projection time-stepping strategies are used in the context of Navier-Stokes as  efficient time integrations. 
The starting point of most projection time-stepping algorithms is Chorin's \cite{Chorin} and Temam's  \cite{Temam-1969}  projection method which consists in decoupling the computation of  the velocity 
field from that of the pressure; in other words, separating the incompressibility constraint from the momentum equation. However, such a strategy needs some elaborations to be applied to the Ginzburg-Landau equations in 
order to segregate the equation for the director field as well. The same difficulty arises in the context of magnetohydrodymanics (MHD) fluids for which  Armero and Simo \cite{Armero-Simo} designed an algorithm which decoupled the computation of the velocity field from the magnetic field. We refer the reader to \cite{Badia-Planas-Gutierrez} where the ideas of Chorin and Temam are combined with the ones of Armero and Simo for the MHD equations. It is in this spirit that the algorithm presented in \cite{minjeaud} is designed for a triphasic Navier-Stokes-Cahn-Hilliard problem, decoupling the Navier-Stokes subproblem from the Cahn-Hilliard one.
 
The goal of this paper is then to extend these types of strategies for developing a numerical scheme for the Ginzburg-Landau equations so that we can decouple the angular, the momentum, and the incompressibility equation.

\section{Finite element approximation}

\subsection{ Preliminaries}

Herein we introduce the hypotheses that will be required along this work.
\begin{enumerate}
\item [(H1)] Let $\Omega$ be bounded domain of $\R^M$ 
 with a polygonal (when $M=2$) or polyhedral (when $M=3$) Lipschitz-continuous boundary.

\item[(H2)] Let $\{{\cal T}_{h}\}_{h>0}$  be a family of quasi-uniform triangulations of 
$\overline{\Om}$ made up of triangles or quadrilaterals in two dimensions 
and tetrahedra or  hexahedra in three dimensions, so that 
$\overline \Omega=\cup_{K\in {\cal T}_h}K$. Further,  let 
${\cal N}_h = \{\a_l\}_{l \in L}$ denote the set of all nodes of ${\cal T}_h$.

\item [(H3)] Conforming finite-element spaces associated with 
${\cal T}_h$ are assumed. In particular, let 
$\mathcal{P}_1(K)$ be
the set of linear polynomials on a finite element $K$. 
Thus the space of continuous, piecewise polynomial 
functions  associated 
with ${\cal T}_h$  is denoted as
$$
X_h = \left\{ v_h \in {C}^0(\overline\Omega) \;:\; 
v_h|_K \in \mathcal{P}_1(K), \  \forall K \in \mathcal{T}_h \right\},
$$
where $\{\phi_i\}_{i=1}^{I}$ stands for its nodal 
basis associated with ${\cal N}_h$. Therefore, 
any element $v_h\in X_h$ can be 
characterized as a vector $V=(V_i)_{i=1}^I 
\in\mathbb{R}^I$ defined as
$$
v_h=\sum_{i=1}^{I} V_i \phi_i.
$$
Moreover, we denote as 
$$
Z_h = \left\{ v_h \in L^{\infty}(\Omega) {\rm \;such
\; that \;} x_h|_K \in \mathcal{P}^1(K) \; \forall K \in
{\cal T}_h \right\}
$$
and
$$
Y_h =\{ w_h \in L^\infty(\Omega) \;:\; w_h|_K \in \mathbb{ R},\ \forall K\in {\cal T}_h\},
$$
where $\{\psi_l = \chi|_{K_l}\}_{l=1}^{L}$ stands for the basis of $Y_h$ associated  with ${\cal T}_h$ and $\chi|_{K_l}$ is the characteristic function of the element $K_l$. Therefore, any element $w_h\in Y_h$ can be  characterized as a vector $W=(W_l)_{l=1}^L \in\mathbb{R}^L$ defined as
$$
w_h=\sum_{l=1}^{L} W_l \psi_l.
$$
The finite-element spaces $\D_h=\X_h$, $\V_h=\X_h\cap\H^1_0(\Omega)$  and  $ P_h=X_h\cap L^2_0(\Omega)$, are used for approximating the director, the velocity and the pressure, respectively.
Additionally,  we select $\W_h=\Y_h$ to be an extra finite-element 
space for an auxiliary variable needed to prove a priori energy estimates for 
Scheme~1 given  below. 
\item [(H4)] Let $0 = t_0 < t_1 < \cdots < t_N = T$ be a uniform partition of the time interval $[0,T]$ so that $k=T/N$ with 
$N\in\mathds{N}$. We suppose that there exist three positive constants $\delta_1$, $\delta_2$ and $\delta_3$, independent of $(h,k,\varepsilon)$,  such that
\begin{equation}\label{Constraint1}
C\frac{k}{h\varepsilon^2}\le  \delta_1,
\end{equation}
\begin{equation}\label{Constraint2}
C \frac{k}{h^{3/2}\varepsilon}\le \delta_2,
\end{equation}
and
\begin{equation}\label{Constraint3}
C\frac{h}{\varepsilon}\le \delta_2,
\end{equation}
where $C>0$ is a constant depending on the data problem but otherwise independent of $(h,k,\varepsilon)$.
\item [(H5)] We suppose that $(\u_0, \d_0)\in \H\times \H^1(\Omega)$ with $|\d_0|=1$ a.e. in $\Omega$.
\end{enumerate} 
\vskip 1cm

Hypothesis $\rm (H3)$ is extremely flexible and allows equal-order 
finite-element spaces for velocity and pressure. Observe that our 
choice of the finite-element spaces for velocity and  pressure does 
not satisfy the discrete inf-sup condition
\begin{equation}\label{LBB-condition}
\|p_h\|_{L^2_0(\Om)}\le \alpha \sup_{\v_h\in\V_h\setminus
\{0\}}\frac{\Big(q_h,\nabla\cdot \v_h\Big)}{\|\v_h\|_{H^1(\Om)}}
\quad \forall\, p_h\in P_h,
\end{equation}
for $\alpha>0$  independent of $h$. 

Some inverse inequalities are established in the next proposition (see \cite{Brenner-Scott}).  
\begin{proposition} Assuming hypotheses $\rm (H1)$-$\rm (H3)$, the following inverse inequalities hold:
\begin{align}
\label{invLinf_L2_disc}
\|\v_h\|_{L^\infty(\Omega)}&\le C_{inv}\,  h^{-3/2} \| \v_h\|, \quad \v_h\in \nabla Z_h,
\\
\label{invLinf_H1}
\| \v_h\|_{L^\infty(\Omega)}&\le C_{inv}\, h^{-1/2} \| \nabla \v_h\|,\quad \v_h\in \V_h,
\end{align}
where $C_{inv}>0$ is a constant independent of $h$. 
\end{proposition}
 
The following proposition is concerned with an interpolation operator $I_h$ associated with the space $\D_h$. In fact, we can think of $I_h $ as the Scott-Zhang interpolation operator, see \cite{Scott-Zhang}.

\begin{proposition} Assuming hypotheses $\rm (H1)$-$\rm (H3)$, there exists $I_h:\H^1(\Omega)\to \D_h$ an interpolation operator 
 satisfying 
\begin{equation}\label{interp_error_Ih}
\|\d-I_h \d\| \leq C_{int}\, h\|\nabla\d\|\quad
\forall\,\d\in\H^1(\Omega),
\end{equation}
and 
\begin{align}
\label{stabLinf}
\|I_h\d\|_{\L^\infty(\Omega)}\le C_{sta} \|\d\|_{\L^\infty(\Omega)} 
\quad\forall\,\d\in \L^\infty(\Omega),
\\
\label{stabH1}
\|I_h\d\|_{\H^1(\Omega)}\le C_{sta} \|\d\|_{\H^1(\Omega)} \quad\forall\,\d\in \H^1(\Omega),
\end{align}
where $C_{int}>0$ and $C_{sta}>0$ are constants independent of $h$. 
\end{proposition}

If $\Pi_0$ is the $L^2$-orthogonal projection operator 
onto $\Y_h$, the following corollary can be deduced from the previous proposition.
\begin{corollary} Assuming hypotheses $\rm (H1)$-$\rm (H3)$, the operator $\Pi_0$ satisfies
\begin{equation}\label{interp_error_Pi0}
\|\d-\Pi_0 \d\| \leq C_{int}\, h\|\nabla\d\|\quad
\forall\,\d\in\H^1(\Omega),
\end{equation}
where $C_{int}$ is a constant independent of $h$.

\end{corollary}

\subsection{The projection time-stepping method}

Next we introduce the ideas that lead to design the numerical 
scheme presented in this work. To obtain a first version of the algorithm for solving \eqref{Ginzburg-Landau-Problem-Truncation}, we split all the differential operator appearing in equation \eqref{GLPT2}.

 
{\sc Scheme 1.}
Let $\u^n, \utilde^n$ and $\d^n$ be given. For $n+1$,
do the following: 
\begin{enumerate}
\item Find $\d^{n+1}$, $\w^{n+1}$,  $\widehat\u^{n+1}$ 
satisfying 
\begin{subequations}\label{scheme1eq1}
\begin{empheq}[left=\empheqlbrace]{align}
\frac{\d^{n+1}-\d^n}{k}+(\widehat\u^{n+1}\cdot\nabla)
\d^n+\gamma\w^{n+1}&=0\mbox{ in }\Omega,\label{scheme1eq1a}\\
-\Delta\d^{n+1}+\widetilde\f_\varepsilon( \d^{n})-\w^{n+1}&=0\mbox{ in }\Omega,\label{scheme1eq1b}\\
\displaystyle
\frac{\widehat\u^{n+1}-\u^{n}}{k}-\lambda (\nabla\d^{n})^T \w^{n+1}&=0\mbox{ in }\Omega,\label{scheme1eq1c}\\
\partial_\n\d^{n+1}&=0\mbox{ on } \partial\Omega.\label{scheme1eq1d}
\end{empheq}
\end{subequations}
\item Find $\widetilde\u^{n+1}$  satisfying
\begin{subequations}\label{scheme1eq2}
\begin{empheq}[left=\empheqlbrace]{align}
\displaystyle
\frac{\widetilde\u^{n+1}-\widehat\u^{n+1}}{k}+(\widetilde\u^n\cdot\nabla)\widetilde\u^{n+1}-\nu \Delta \widetilde\u^{n+1}&=0
\mbox{ in }\Omega,\label{scheme1eq2a}\\
\widetilde\u^{n+1}&=0\mbox{ on } \partial\Omega.
\label{scheme1eq2b}
\end{empheq}
\end{subequations}
\item Find $\u^{n+1}$ and $p^{n+1}$ satisfying 
\begin{subequations}\label{scheme1eq3}
\begin{empheq}[left=\empheqlbrace]{align}
\displaystyle
\frac{\u^{n+1}-\widetilde\u^{n+1}}{k}+\nabla p^{n+1}&=
0\mbox{ in } \Omega,\label{scheme1eq3a}\\
\nabla\cdot\u^{n+1}&=0\mbox{ in }\Omega,\label{scheme1eq3b}\\
\u^{n+1}\cdot\n&=0\mbox{ on } \partial\Omega.\label{scheme1eq3c}
\end{empheq}
\end{subequations}
\end{enumerate}

The time discretization of the nonlinear terms 
is exactly the same as the one implemented in \cite{Guillen-Gutierrez} which is on 
spirit of a linearization.
Observe that $\widehat\u^{n+1}$ depends only on $\u^n$, $\nabla \d^{n}$ 
and $\w^{n+1}$ so we can avoid computing it if we add equation $\eqref{scheme1eq1a}$ and equation 
$\eqref{scheme1eq1c}$. 
Moreover, as usual for developing projection algorithms, the divergence operator is applied to equation 
$\eqref{scheme1eq3a}$ in order to decouple the computations of  $\utilde^{n+1}$ and $p^{n+1}$. Thus Scheme 1 remains as follows. 

{\sc Scheme 2.}
Let $\u^n$ and $\d^n$ be given. For $n+1$,
do the following steps: 
\begin{enumerate}
\item Find $\d^{n+1}$, $\w^{n+1}$ satisfying  
\begin{subequations}\label{scheme2eq1}
\begin{empheq}[left=\empheqlbrace]{align}
\displaystyle
\frac{\d^{n+1}-\d^n}{k}+([\u^{n}+k\lambda (\nabla\d^{n})^T \w^{n+1}]\cdot\nabla)\d^n+
\gamma\w^{n+1}&=0\mbox{ in } \Omega,\label{scheme2eq1a}\\
-\Delta\d^{n+1}+\widetilde \f_\varepsilon(\d^{n})
-\w^{n+1}&=0\mbox{ in }\Omega,
\label{scheme2eq1b}\\
\partial_\n\d^{n+1}&=0\mbox{ on } \partial\Omega.\label{scheme2eq1c}
\end{empheq}
\end{subequations}
\item Find $\widetilde\u^{n+1}$ satisfying 
\begin{subequations}\label{scheme2eq2}
\begin{empheq}[left=\empheqlbrace]{align}
\displaystyle
\frac{\utilde^{n+1}-\u^{n}}{k}+(\utilde^n\cdot\nabla)\utilde^{n+1}-
\nu \Delta \utilde^{n+1}&=\lambda (\nabla\d^{n})^T \w^{n+1}\mbox{ in } \Omega,\label{scheme2eq2a}\\
\utilde^{n+1}&=0\mbox{ on } \partial\Omega.\label{scheme2eq2b}
\end{empheq}
\end{subequations}
\item Find $p^{n+1}$  satisfying     
\begin{subequations}\label{scheme2eq3}
\begin{empheq}[left=\empheqlbrace]{align}
\displaystyle
-\Delta p^{n+1}&=-\displaystyle\frac{1}{k} 
\nabla\cdot\utilde^{n+1}\mbox{ in } \Omega,\label{scheme2eq3a}\\
\partial_\n p^{n+1}&=0\mbox{ on } \partial\Omega.\label{scheme2eq3b}
\end{empheq}
\end{subequations}
\item Compute $\u^{n+1}$ as
\begin{equation}\label{scheme2eq4}
\u^{n+1}=\utilde^{n+1}-k\nabla p^{n+1}\mbox{ in } \Omega.
\end{equation}
\end{enumerate}

It is well to point out at this level that Scheme~2
decouples the computation of the pair 
$(\d^{n+1}, \w^{n+1})$, the intermediate velocity 
$\widetilde\u^{n+1}$, the pressure $p^{n+1}$
 and the end-of-step velocity $\u^{n+1}$. The auxiliary 
 variable $\w^{n+1}$ is totally artificial and is 
 only introduced to help us to prove a priori energy 
 estimates.  This issue will treat in 
 detail in Section \ref{sect:Implementation} when we 
 set up the algebraic version for Scheme 2, once its 
 finite-element counterpart is performed just below.  
 Analogously, the end-of-step velocity $\u^{n+1}$ can 
 be entirely brought out from Scheme 2 by simple 
 computations but it again is desirable to leave it in 
 order to establish energy bounds.         
     
The numerical scheme under consideration is based on 
the weak formulation of Scheme 2 by using the 
finite-element spaces defined in Hypothesis (H3) and a 
spatial stabilization technique in order to use the 
same interpolation for velocity and pressure. Thus 
we have: 
  
{\sc Scheme 3}
 
Let $(\d^{n}_h, \u^{n}_h)\in \D_h\times \V_{h}$ be 
given. For $n+1$, do the following steps: 
\begin{enumerate}
\item Find $(\d^{n+1}_h, \w^{n+1}_h)\in \D_h\times\W_h$ satisfying 
\begin{subequations}\label{scheme3eq1}
\begin{empheq}[left=\empheqlbrace]{align}
\displaystyle \left(\frac{\d^{n+1}_h-\d^n_h}{k},\bar\w_h\right) + 
(( \utecho_h\cdot \nabla)\d^n_h, \bar\w_h)+
\gamma( \w^{n+1}_h, \bar\w_h)=0,\label{scheme3eq1a}\\
(\nabla \d^{n+1}_h, \nabla \bar  \d_h)+
(\widetilde\f_\varepsilon(\d^n_h),\bar \d_h)-
( \w^{n+1}_h, \bar\d_h)=0,\label{scheme3eq1b}
\end{empheq}
\end{subequations}
for all $(\bar\d_h, \bar\w_h)\in \D_h\times\W_h$, where 
\begin{equation}\label{scheme3eq1c}
\utecho_h=\u^n_h+\lambda\, k (\nabla\d_h^n)^T \w^{n+1}_h.
\end{equation}

\item Find $\widetilde\u_n^{n+1}\in \V_h$ satisfying 
\begin{equation}\label{scheme3eq2}
\displaystyle \left(\frac{\utilde^{n+1}_h-\u^n_h}{k},\bar\u_h\right) +
c(\utilde^n_h, \utilde^{n+1}_h,\bar\u_h)
+\nu(\nabla\utilde^{n+1}_h,\nabla\bar\u_h)-
\lambda ((\nabla\d^n_h)^T \w^{n+1}_h, \bar\u_h )= 0, 
\end{equation}
for all $\bar\u_h\in \V_h.$
\item Find $p^{n+1}_h\in  P_{h}$ satisfying
\begin{equation}\label{scheme3eq3a}
k(\nabla p^{n+1}_h, \nabla \bar  p_h)+ j (p^{n+1}_h, \bar p_h)
+(\nabla\cdot\widetilde\u^{n+1}_h,\bar p_h)=0,
\end{equation}
for all $\bar p_h\in P_h$, with
\begin{equation}\label{scheme3eq3b}
j(p^{n+1}_h, \bar p_h)=\left\{
\begin{array}{rcl}
\tau (\nabla p^{n+1}_h-\Pi_1(\nabla p^{n+1}_h), \nabla \bar p_h-
\Pi_1(\nabla \bar p_h)), & \mbox{ with }&\displaystyle \tau=S 
\frac{h^2}{\nu},\\
\\
\tau ( p^{n+1}_h-\Pi_0( p^{n+1}_h), \bar p_h-\Pi_0( \bar p_h)), & 
\mbox{ with }& \displaystyle \tau=S \frac{1}{\nu},
\end{array}
\right.
\end{equation}
where $S$ is an algorithmic constant, and $\Pi_1$ and $\Pi_0$  are the 
$L^2$-orthogonal projection operator onto $\X_h$ and $\Y_h$, respectively.
\item Compute $\u^{n+1}_h\in \V_h+\nabla P_h$ as 
\begin{equation}\label{scheme3eq4}
\u^{n+1}_h=\widetilde\u^{n+1}_h- k \nabla p^{n+1}_h.
\end{equation}
\end{enumerate}

To ensure the skew-symmetric of the trilinear 
convective term in \eqref{scheme3eq2}, we have defined
$$
c(\u_h,\v_h,\w_h)=((\u_h\cdot\nabla\v_h), \w_h)+
\frac{1}{2} (\nabla\cdot \u_h, \v_h\cdot	\w_h)
$$   
for all $\u_h, \v_h, \w_h\in \V_h$. Thus 
$c(\u_h,\v_h,\v_h)=0$ for all $\u_h, \v_h\in \V_h$.  

The idea for the stabilization term $j(\cdot, \cdot )$ 
in \eqref{scheme3eq3b} is to penalize either the difference 
between the pressure gradient and its projection onto the space of piecewise linear, continuous functions or the difference between the pressure and its 
projection onto the space of piecewise constant functions. The former was 
initially proposed in \cite{Codina-Blasco} for the 
Stokes problem and extended later for the Navier-Stokes 
problem in \cite{Codina} together with a projection 
time-stepping method. The latter was instead proposed 
for the Stokes problem in \cite{dohrmann}. 
They both can be motivated since the pressure stability
 provides for the crude projection time-stepping method 
 depends on the time-step size 
$k$ so that lost of pressure stability is expected 
under Hypothesis $\rm {(H4)}$ since $k$ must be 
considerably small. The reader can refer to 
\cite{Burman} for a general description of the 
stabilization technique. 

For suitable initial approximations 
$(\u_{0h}, \d_{0h})$ of $(\u_0, \d_0)$, we  consider 
$\d_{0h}\in\D_h$  such that   
\begin{equation}\label{initial-d0}
\d_{0h}=I_h\d_0
\end{equation}
and  $(\u_{0h},p_{0h})\in \V_h\times P_h$ such that
\begin{subequations}\label{initial-u0}
\begin{empheq}[left=\empheqlbrace]{align}
(\u_{0h}, \bar \u_h)+(\nabla p_{0h}, \bar\u_h )&=(\u_0, \bar\u_h ),
\label{initial-u0eq1}\\
(\nabla\cdot\u_h, \bar p_h)+j( p_{0h}, \bar p_h)&=0,\label{initial-u0eq2}
\end{empheq}
\end{subequations}
for all $(\bar\u_h, \bar p_h)\in \V_h\times P_h$.

In what follows we prove the existence and uniqueness of a solution to Scheme 3. Since Scheme 3 is a finite dimensional system having the same number of unknowns as equations, existence and uniqueness are equivalent. Let $\delta \d_h^{n+1}$ and $\delta\w^{n+1}_h$ denote the difference between two possible solutions to (\ref{scheme3eq1}). It not hard to check that $\delta\d^{n+1}_h$ and $\delta\w^{n+1}_h$ satisfy 
\begin{subequations}\label{scheme3eq1bis}
\begin{empheq}[left=\empheqlbrace]{align}
\displaystyle \frac{1}{k}(\delta\d^{n+1}_h,\bar\w_h) + 
((\nabla\d^n_h)^T\delta\w^{n+1}_h, (\nabla\d^n_h)^T \bar\w_h)+
\gamma( \delta\w^{n+1}_h, \bar\w_h)=0,\label{scheme3eq1abis}
\\
(\nabla \delta\d^{n+1}_h, \nabla \bar  \d_h)-
( \delta\w^{n+1}_h, \bar\d_h)=0,\label{scheme3eq1bbis}
\end{empheq}
\end{subequations}
for all $(\bar\d_h, \bar\w_h)\in \D_h\times\W_h$. Taking $\bar\w_h=\delta\w^{n+1}_h$ and $\bar\d_h=\delta\d^{n+1}_h$ into (\ref{scheme3eq1abis}) and (\ref{scheme3eq1bbis}), respectively, we have 
$$
\frac{1}{k} \|\nabla\delta\d^{n+1}_h\|^2+\|(\gamma+\|(\nabla\d^{n}_h)^T\|^2) \delta\w^{n+1}_h\|^2=0.
$$  
This implies that $\delta\w^{n+1}_h=\boldsymbol{0}$ and $\delta\d^{n+1}_h=\d$ with $\d=(d_1, \cdots, d_M)^T\in \R^M$. As now $\delta \d^{n+1}_h=\d\in \W_h$, we are allowed to take $\bar\w_h=\delta\d_h^{n+1}$ in \eqref{scheme3eq1abis}  to find 
$\|\delta\d^{n+1}_h\|^2=0$. Therefore $\delta\d^{n+1}_h=\boldsymbol{0}$. Thus we have proved uniqueness of a solution to \eqref{scheme3eq1}. Analogously, we can prove the existence and uniqueness of a solution to \eqref{scheme3eq2}, (\ref{scheme3eq3a}), and \eqref{scheme3eq4}.
  
\section{A priori energy estimates}

This section is devoted to proving, by induction on $n$,
a priori energy estimates for Scheme 3. Let us denote 
$$
{\cal E}(\u_h, \d_h)=\frac{1}{2}\|\u_h\|^2+\frac{\lambda}{2} \|\nabla\d_h\|^2+ 
\lambda\int_{\Omega} \widetilde F_{\varepsilon}(\d_h),
$$
for any $(\u_h, \d_h)\in \V_h\times \D_h$. The term  ${\cal E}(\u_h, \d_h)$ represents 
the total energy involving in Scheme~3 which are the kinetic energy $\displaystyle\frac{1}{2}\|\u_h\|^2$, 
the elastic energy $\displaystyle\frac{\lambda}{2}\|\nabla\d_h\|^2$ and the penalty energy 
$\displaystyle\lambda\int_{\Omega} F_\varepsilon(\d_h)$. 

\begin{lemma}\label{le:induction}
Assume that hypotheses $\rm(H1)$-$\rm(H4) $ hold and that there 
exists a constant $C_0>0$,  independent of $h$,
$k$, and $\varepsilon$, such that
\begin{equation}\label{Bound-n}
{\cal E}(\u^n_h, \d^n_h)\leq C_0.
\end{equation}
Then, for $k$, $h$ and $\varepsilon$ small enough (depending on 
$C_0$ and independent of $n$),  the corresponding solution
$(\u^{n+1}_h,\d^{n+1}_h,\w^{n+1}_h)$ to Scheme
3 satisfies the following inequality:
\begin{equation}\label{induction}
\displaystyle
{\cal E}(\u^{n+1}_h,\d^{n+1}_h)-{\cal E}(\u^{n}_h,\d^{n}_h) 
+\frac{k}{2}\left(\nu\|\nabla\widetilde\u_h^{n+1}\|^2 
+\lambda\gamma \|\w^{n+1}_h\|^2\right)\le 0.
\end{equation}

\end{lemma}
\begin{proof} 
By taking $\bar p_h=p^{n+1}_h$ as a test function  into \eqref{scheme3eq3a}
and taking into account \eqref{scheme3eq4},
we have 
\begin{equation*}
j(p^{n+1}_h, p^{n+1}_h)=(\widetilde \u^{n+1}_h-k\nabla p^{n+1}_h, 
\nabla p^{n+1}_h )=(\u^{n+1}_h, \nabla p^{n+1}_h ).
\end{equation*}
Hence, from \eqref{scheme3eq4}, we obtain
\begin{equation*}
\frac{1}{2}\|\widetilde \u^{n+1}_h\|^2 = \frac{1}{2}\|\u^{n+1}_h\|^2+
\frac{k^2}{2}\|\nabla p^{n+1}_h\|^2+k\,j(p^{n+1}_h, p^{n+1}_h).
\end{equation*}
Next, by replacing $\u^{n}_h$, from \eqref{scheme3eq1c}, into  
\eqref{scheme3eq2} and taking $\bar\u_h=2\,k\, \widetilde\u^{n+1}_h$ as a test 
function into \eqref{scheme3eq2}, we obtain 
\begin{equation}\label{aux1}
\displaystyle \frac{1}{2}\|\utilde^{n+1}_h\|^2
-\frac{1}{2}\|\utecho_h\|^2
+\frac{1}{2}\|\utilde_h^{n+1}-\utecho_h\|^2
+\nu\,k \|\nabla \utilde_h^{n+1}\|^2= 0.
\end{equation}
By combining these last two equations, we find
\begin{align}
 \frac{1}{2}\|\u^{n+1}_h\|^2 - \frac{1}{2}\|\utecho_h\|^2
 +\nu\,k \| \nabla\utilde_h^{n+1}\|^2
+\frac{1}{2}\|\utilde_h^{n+1}-\utecho^n_h\|^2
+\frac{k^2}{2}\|\nabla p^{n+1}_h\|^2+k\,j(p^{n+1}_h, p^{n+1}_h)= 0.
\end{align}
Now, if we consider $\bar\w_h=\lambda\,k\,\w^{n+1}_h$ into
\eqref{scheme3eq1a} jointly with $\bar\d_h=\lambda(\d^{n+1}_h-\d^n_h)$ 
into \eqref{scheme3eq1b}, we arrive at
\begin{equation}\label{aux5}
\begin{array}{l}
\displaystyle
\frac{\lambda}{2}\left(\|\nabla\d^{n+1}_h\|^2-\|\nabla\d^n_h\|^2
+\|\nabla(\d^{n+1}_h-\d^{n}_h)\|^2\right)
+\lambda\,\gamma \,k\|\w^{n+1}_h\|^2\\
\hspace{2cm}+\lambda(\d^{n+1}_h-\d^n_h,\widetilde\f_\varepsilon(\d^{n}_h))
+\lambda\, k ((\utecho_h\cdot\nabla)\d^{n}_h,\w^{n+1}_h)=0.
\end{array}
\end{equation}
In view of (\ref{scheme3eq1c}),
a simple calculation shows that
$$
\frac{1}{2}\|\utecho_h\|^2-\frac{1}{2}\|\u^n_h\|^2
+\frac{1}{2}\|\utecho_h-\u^n_h\|^2 
-\lambda\, k ((\nabla\d_h^n)^T \w^{n+1}_h, \utecho_h)=0. 
$$
This equality together with  \eqref{aux1}, \eqref{aux5},  and the fact that
$$
-((\nabla\d^n_h)^T\w^{n+1}_h,\utecho_h)+
((\utecho_h\cdot\nabla)\d^{n}_h,\w^{n+1}_h)=0
$$ 
implies that 
\begin{eqnarray}
\displaystyle\frac{1}{2}\left(\|\u^{n+1}_h\|^2
+\lambda\|\nabla\d_h^{n+1}\|^2\right)
-\frac{1}{2}(\|\u^n_h\|^2+\lambda\|\nabla\d^n_h\|^2 )
+k(\nu \|\nabla\widetilde\u^{n+1}_h\|^2 
+\gamma\lambda\|\w^{n+1}_h\|^2)&& \nonumber\\
\hspace{2cm}\displaystyle+\frac{1}{2}\|\utilde_h^{n+1}-\utecho^n_h\|^2
+\frac{1}{2}\|\utecho_h-\u^n_h\|^2
+\frac{\lambda\, k^2}{2}\|\nabla\delta_t\d^{n+1}\|^2
+\frac{k^2}{2}\|\nabla p^{n+1}_h\|^2&&\label{aux6}\\
\hspace{4cm}+k\,j(p^{n+1}_h, p^{n+1}_h)
+\lambda(\d^{n+1}_h-\d^{n}_h,\widetilde\f_{\varepsilon}(\d^n_h))
&=& 0.\nonumber
\end{eqnarray}
What remains is to control the term 
$\lambda (\d^{n+1}_h-\d^{n}_h,\widetilde \f_{\varepsilon}(\d^n_h))$. 
By using the Taylor polynomial of order $2$ of  
$\widetilde F_\varepsilon$ with respect to $\d^n_h$ 
evaluated at $\d^{n+1}_h$, it gives 
$$
\widetilde F_\varepsilon(\d^{n+1}_h)-\widetilde F_\varepsilon(\d^n_h)=
\nabla_\d \widetilde F_\varepsilon(\d^n_h) (\d_h^{n+1}-\d^n_h)+(\d_h^{n+1}-\d^n_h)^TH_\d \widetilde F_\varepsilon(\d^n_{h,\theta})(\d_h^{n+1}-\d^n_h),
$$
where $\d^n_{h,\theta}=\theta \d^{n+1}_h+(1-\theta)\d^n_h$ for some 
$\theta\in (0,1)$, and  $H_\d\widetilde F_\varepsilon$ stands for the Hessian matrix of
$\widetilde F_\varepsilon$ with 
respect to $\d$. It is important to emphasize that 
$\|H_\d  \widetilde F_\varepsilon(\d^n_{h,\theta})\|_{L^\infty(\Omega)}
\le C/\varepsilon^2$. Thus, integrating over $\Omega$ and inserting 
it  into \eqref{aux1},  we obtain
\begin{eqnarray}
{\cal E}(\u^{n+1}_h, \d^{n+1}_h)- {\cal E}(\u^{n}_h, \d^{n}_h) 
+ k(\nu \| \nabla \utilde_h^{n+1}\|^2 +\lambda\gamma \|\w^{n+1}_h\|^2) 
+\frac{1}{2}\|\utilde_h^{n+1}-\utecho_h\|^2
+\frac{1}{2}\|\utecho_h-\u^n_h\|^2&&\nonumber \\
+\frac{\lambda\, k^2}{2}\|\nabla\delta_t\d^{n+1}_h\|^2
+\frac{k^2}{2}\|\nabla p^{n+1}_h\|^2+k\,j(p^{n+1}_h, p^{n+1}_h) \le C \frac{\lambda}{\varepsilon^2} k^2\|\delta_t\d^{n+1}_h\|^2:=I.&&
\label{interm-estim}
\end{eqnarray}
Next we take $\bar\w_h=\Pi_0(\bar\w)$ as a test function into
\eqref{scheme3eq1a} with $\bar\w\in \L^2(\Omega)$:
\begin{equation}\label{aux7}
\left(\delta_t\d^{n+1}_h, \Pi_0( \bar\w)\right)=
-((\utecho_h\cdot\nabla)\d^{n}_h, \Pi_0(\bar\w))-\gamma(\w^{n+1}_h, \bar\w).
\end{equation} 
Note that $\Pi_0$ has been neglected from the second term on the right-hand side of 
 \eqref{aux7}. Using \eqref{Bound-n}, which implies 
$\frac{\lambda}{2}\|\nabla\d^n_h\|^2\le C_0$, we  estimate the right-hand 
side of \eqref{aux7}  as
$$
((\utecho_h\cdot\nabla)\d^{n}_h, \Pi_0(\bar\w))\le 
\|\utecho_h\|_{\L^\infty(\Omega)}\|\nabla\d^n_h\| \| \Pi_0(\bar\w)\|\le 
\sqrt{\frac{2C_0}{\lambda}}\|\utecho_h^n\|_{\L^\infty(\Omega)} \|\bar\w\| 
$$
and
$$
(\w^{n+1}_h, \bar\w)\le \|\w^{n+1}_h\|\|\bar\w\|.
$$
We find, after applying a duality argument, that   
$$
\| \Pi_0 (\delta_t\d^{n+1}_h)\|
\le C (\|\utecho_h\|_{\L^\infty(\Omega)} +\gamma\|\w^{n+1}_h\|
). 
$$
The triangle inequality together with (\ref{interp_error_Pi0})  gives 
$$
\begin{array}{rcl}
\|\delta_t\d^{n+1}_h\|&\le& \|\Pi_0(\delta_t\d^{n+1}_h)\|+\|\delta_t\d^{n+1}_h-\Pi_0(\delta_t\d^{n+1}_h)\|
\\
&\le & C (\|\utecho_h\|_{\L^\infty(\Omega)} +\gamma\|\w^{n+1}_h\|)+h\|\nabla(\delta_t\d^{n+1}_h)\|.
\end{array}
$$
Therefore, the bound of $I$ remains as
$$
I\le C\frac{\lambda}{\varepsilon^2}k^2 \left(
\|\utecho_h\|^2_{\L^\infty(\Omega)} +\gamma \|\w^{n+1}_h\|^2\right)+C \lambda \frac{k^2h^2}{\varepsilon^2}\|\nabla\delta_t\d^{n+1}\|^2.
$$ 
The triangle inequality gives    
$$
\|\utecho_h\|_{\L^\infty(\Omega)}\le\|\utilde^{n+1}_h\|_{\L^\infty(\Omega)}
+ \|\utilde^{n+1}_h-\utecho_h\|_{\L^\infty(\Omega)},
$$
which combined with the two inverse inequalities \eqref{invLinf_L2_disc} and \eqref{invLinf_H1}  provides
$$
\|\utecho_h
\|_{\L^\infty(\Omega)}\le \frac{ C}{ h^{1/2}}  \|\nabla
\utilde^{n+1}_h\|+ \frac{C}{h^{3/2}} \|\utilde^{n+1}_h-\utecho_h\|.
$$
Finally, $I$ remains bounded as
$$
I\le C\frac{k^2}{\varepsilon^2}  \left( \frac{1}{h} 
\|\nabla \utilde^{n+1}_h\|^2+ \frac{1}{h^{3}}\|\utilde^{n+1}_h-\utecho_h\|^2
+\lambda\,\gamma \|\w^{n+1}_h\|^2\right)+C \lambda\frac{k^2h^2}{\varepsilon^2}\|\nabla\delta_t\d^{n+1}_h\|^2.
$$ 
Adjusting the constants $\delta_1$ and $\delta_1$ from constraints \eqref{Constraint1} and \eqref{Constraint2}, we get
$$
I\le \frac{k}{2} \left(\nu \|\nabla\utilde^{n+1}_h\|^2
+\lambda\,\gamma\,\|\w^{n+1}_h\|^2\right)
+ \frac{1}{2}\|\utilde^{n+1}_h-\utecho_h\|^2+ \lambda\frac{k ^2}{4}\|\nabla\delta_t\d^{n+1}_h\|^2.
$$
To conclude, we  obtain inequality \eqref{induction} 
by using the above estimate in \eqref{interm-estim}.
\end{proof}

In order to initialize the induction argument on $n$ we need to insure the existence of $C_0>0$ in the requirement \eqref{Bound-n} for $n=0$.
\begin{lemma}\label{le:initial-bound}
 Assuming hypotheses $\rm (H1)$-$\rm (H5)$, then there exists a constant $C_0>0$, independent of $h$, $k$, and $\varepsilon$, such that 
\begin{equation}\label{initial-bound}
{\cal E}(\u_{0h}, \d_{0h})\leq C_0
\end{equation}
for the initial approximations $(\u_{0h}, \d_{0h})$ defined in 
\eqref{initial-d0} and \eqref{initial-u0}.
\end{lemma}
\begin{proof} We take $\bar\u=\u_{0h}$ and $\bar p_h=p_{0h}$ as 
test functions into \eqref{initial-u0} to obtain
\begin{equation}\label{aux8}
\frac{1}{2}\|\u_{0h}\|^2+j(p_{0h}, p_{0h})\le \frac{1}{2}\|\u_0\|^2.
\end{equation}

Moreover, from \eqref{stabH1}, we have
\begin{equation}\label{aux9}
\|\d_{0h}\|_{\H^1(\Omega)}\le \|\d_0\|_{\H^1(\Omega)}.
\end{equation}
Observe that $\widetilde F_\varepsilon(\d)\le F_\varepsilon(\d)$ holds. Now, we bound as in \cite{Guillen-Gutierrez}, 
\begin{eqnarray}
\displaystyle
\int_{\Omega}\widetilde F_\varepsilon(\d_{0h})
\le \int_{\Omega} F_\varepsilon(\d_{0h})&\le&
\displaystyle \frac{1}{ \varepsilon^2}\int_{\Omega} (|\d_{0h}|^2-|\d_0|^2)^2
=\frac{1}{ \varepsilon^2}\int_{\Omega}(|\d_{0h}+\d_0||\d_{0h}-\d_{0}|)^2
\nonumber\\
&\le & \displaystyle \frac{1}{ \varepsilon^2}\|\d_{0h}
+\d_0\|^2_{\L^\infty(\Omega)} \|\d_{0h}-\d_0\|^2\le 
C \frac{h^2}{\varepsilon^2}\|\d_0\|^2_{\H^1(\Omega)}, \label{aux10}
\end{eqnarray}
where \eqref{stabLinf}, \eqref{interp_error_Ih} 
and  constraint \eqref{Constraint3} has been applied. Combining (\ref{aux10}) with (\ref{aux8}) and (\ref{aux9}), we obtain \eqref{initial-bound}.
\end{proof}

We are now ready to prove the a priori global-in-time energy estimates 
for Scheme 3.
\begin{theorem}\label{Th:bounds}
Assume that $\rm (H1)$-$\rm (H5)$ hold. Then  there exist $h_0$, $k_0$, and $\varepsilon_0$ small enough so that the corresponding solution
to Scheme~3 satisfies  the following global-in-time discrete energy inequality: 
\begin{equation}\label{global-estimate}
\max_{r\in\{0,\cdots, N-1\}}\left\{{\cal E}(\u^{r+1}_h,\d^{r+1}_h, p_h^{r+1})
 +\frac{k}{2}\sum_{n=0}^r \left(\nu
\|\nabla\widetilde\u_h^{n+1}\|^2 +\lambda\gamma \|\w^{n+1}_h\|^2\right)\right\}
\le {\cal E}(\u_{0h},\d_{0h}, p_{0h}).
\end{equation}
\end{theorem}
\begin{proof} For $r=0$, inequality \eqref{global-estimate} holds 
from \eqref{induction} in Lemma \ref{le:induction} since  \eqref{Bound-n} 
holds by Lemma \ref{le:initial-bound}.  Next we suppose that inequality 
\eqref{global-estimate} holds for $r=n-1$. Then we have that 
$$
{\cal E}(\u_h^n,\d^n_h)\le {\cal E}(\u_{0h},\d_{0h} )\le C_0.
$$
Thus, inequality \eqref{global-estimate} holds for $r=n$ 
from \eqref{induction} in Lemma \ref{le:induction}.      
\end{proof}

\section{Implementation strategy}
\label{sect:Implementation}
Our goal of this section is to discuss some efficient numerical implementations of Scheme~3. Basically, we want to avoid computing the end-of-step velocity $\utilde_h^{n}$ \cite{Guermond_1996, Guermond-Miner-Shen} and the auxiliary director variable $\w^{n+1}_h$.  The alternative realization of Scheme 3 will be carried out in three steps as follows.    

Firstly, Scheme~3 is rewritten in terms of the intermediate velocity $\utilde_h^{n}$ only by eliminating the end-of-step velocity $\u_h^n$, by means of equation \eqref{scheme3eq4}. Thus we arrive at the following modified version of Scheme 3:

{\sc Scheme 4.}

Let $(\d^{n}_h, \utilde^{n}_h)\in \D_h\times \V_{h}$ be given. For $n+1$, do the following steps: 
\begin{enumerate}
\item Find $(\d^{n+1}_h, \w^{n+1}_h)\in \D_h\times\W_h$ 
satisfying 
\begin{subequations}\label{scheme4eq1}
\begin{empheq}[left=\empheqlbrace]{align}
\displaystyle \left(\frac{\d^{n+1}_h-\d^n_h}{k},\bar\w_h\right) + 
\lambda\, k\,( ( (\nabla\d_h^n)^T \w^{n+1}_h
\cdot \nabla)\d^n_h, \bar\w_h)\nonumber \\
+
\gamma( \w^{n+1}_h, \bar\w_h)=-(((\utilde^n_h-k\nabla p_h^n)\cdot \nabla)\d^n_h, \bar\w_h ),\label{scheme4eq1a}\\
(\nabla \d^{n+1}_h, \nabla \bar  \d_h)+( \w^{n+1}_h, \bar\d_h)=-(\widetilde\f_\varepsilon(\d^n_h),\bar \d_h),\label{scheme4eq1b}
\end{empheq}
\end{subequations}
for all $(\bar\d_h, \bar\w_h)\in \D_h\times\W_h$.

\item Find $\widetilde\u_n^{n+1}\in \V_h$ satisfying 
\begin{equation}\label{scheme4eq2}
\displaystyle \left(\frac{\utilde^{n+1}_h-\utilde^n_h}{k},\bar\u_h\right) +
c(\widetilde\u^n_h, \widetilde\u^{n+1}_h,\bar\u_h)
+\nu(\nabla\utilde^{n+1}_h,\nabla\bar\u_h)
=-(\nabla p_h^n, \bar\u_h )+
\lambda ((\nabla\d^n_h)^T \w^{n+1}_h, \bar\u_h ), 
\end{equation}
for all $\bar\u_h\in \V_h.$ Due to the semi-explicit treatment of the convective term, each velocity component can be computed in a parallel machine. Thus, a large amount of computer memory and time can be saved.
\item Find $p^{n+1}_h\in  P_{h}$ satisfying
\begin{equation}\label{scheme4eq3a}
k(\nabla p^{n+1}_h, \nabla \bar  p_h)+ j (p^{n+1}_h, \bar p_h)
=-(\nabla\cdot\widetilde\u^{n+1}_h,\bar p_h),
\end{equation}
for all $\bar p_h\in P_h$, where 
$j(p_h, \bar p_h)$ is defined in \eqref{scheme3eq3b}.
\end{enumerate}

Secondly, let $N_u = \dim(\V_h)$, $N_p = \dim(P_h)$, $N_d=\dim(\D_h)$, and $N_w=\dim(\W_h)$ and let $\{\phi^u_i\}_{i=1}^{N_u}$, $\{\phi^p_i\}_{i=1}^{N_p}$, $\{\phi^d_i\}_{i=1}^{N_d}$, $\{\psi^w_i\}_{i=1}^{N_w}$ be finite-element bases  for $\V_h$, $P_h$, $\D_h$, and $\W_h$,  respectively, constructed from $\{\phi_i\}_{i=1}^{I}$ and $\{\psi_l\}_{l=1}^{L}$ for $X_h$ and $Y_h$, respectively, given in $\rm( H3)$. Thus, define the following matrices. 
For $\W_h$, we have: 
$$
\Mmat_{d,w}=
\left(\int_\Omega \phi_i^d\cdot\phi_j^w\right),\quad
\displaystyle\Cmat_{w}=
\left(\int_\Omega ((\nabla\d_h^n)^T 
\phi_i^w\cdot\nabla)\d^n_h\cdot\phi_j^w\right),
\quad
\Mmat_{w}=
\left(\int_\Omega \phi_i^w\cdot\phi_j^w\right).
$$For $\D_h$, we have:
$$
\Mmat_{w,d}=
\left(\int_\Omega \phi_i^w\cdot\phi_j^d\right),\quad
\Lmat_d=
\left(\int_\Omega \nabla\phi_i^d\cdot\nabla\phi_j^d\right).
$$
For $\V_h$, we have:
$$
\Mmat_u=
\left(\int_\Omega \phi_i^u\cdot\phi_j^u\right), \quad\Lmat_u=
\left(\int_\Omega \nabla\phi_i^u\cdot\nabla\phi_j^u\right),
\quad
\Emat_{u}=
\left(\int_\Omega (\nabla\d_h^n)^T\psi_i^w\cdot\phi_j^u\right),
$$
$$
\Cmat_{u}=
\left(\int_\Omega (\u_h^n\cdot\nabla)\phi_i^u\cdot\phi_j^u+
\frac{1}{2}\int_\Omega (\nabla\cdot\u_h^n)\phi_i^u\cdot\phi_j^u\right).
$$
For $P_h$, we have:
\begin{equation*}
\displaystyle\Lmat_{p}=
\left(\int_\Omega \nabla\phi_i^p\cdot\nabla\phi_j^p\right), \quad \Jmat_p=j(\phi_i^p, \phi_j^p).
\end{equation*}
Moreover, let us denote by $\Wvec \in \mathbb{R}^{N_w}$, $ \Dvec \in \mathbb{R}^{N_d}$, $\widetilde\Uvec \in \mathbb{R}^{N_u}$, and $ \Pvec \in \mathbb{R}^{N_p}$
the coordinate vectors of the finite-element functions $\w \in \W_h$,  $\d \in \D_h$, $\u\in \V_h$, and $p\in P_h$, respectively.

By using these ingredients, we can write the matrix form of Scheme~4 as follows.

{\sc Matrix version of Scheme~4}

Let  $\Dvec^{n}\in \R^{N_d}$ and $\widetilde\Uvec^{n}\in\R^{N_u}$ be given. For $n+1$,  compute the following steps:
\begin{enumerate}
\item Find $\Dvec^{n+1}\in \R^{N_d}$ and $ \Wvec^{n+1}\in\R^{N_w}$ 
by solving 
\begin{subequations}\label{scheme5eq1}
\begin{empheq}[left=\empheqlbrace]{align}
\displaystyle \frac{1}{k}\Mmat_{d,w}\Dvec^{n+1}+ 
\left(\lambda k\Cmat_{w} +\gamma\Mmat_w\right)\Wvec^{n+1}&=\frac{1}{k}\Mmat_{d,w}\Dvec^{n} -\Fvec_{w},\label{scheme5eq1a}\\
-\Lmat_d\Dvec^{n+1}+\Mmat_{w,d}\Wvec^{n+1}&=\Fvec_\varepsilon,
\label{scheme5eq1b}
\end{empheq}
\end{subequations}
where $\displaystyle\Fvec_{w}\in \R^{N_d}$ and $\Fvec_\varepsilon\in \R^{N_w}$ defined, respectively, as
\begin{equation}\label{scheme5eq1c}
\displaystyle\Fvec_{w}=
\left(\int_\Omega ((\utilde_h^n
-k\nabla p_h^n)\cdot\nabla)\d^n_h\cdot\psi_j^w\right),\quad\mbox{ and }\quad
\displaystyle\Fvec_\varepsilon=
\left(\int_\Omega \widetilde\f_\varepsilon(\d^n_h)\cdot\phi_j^d\right). 
\end{equation}

\item Find $\widetilde\Uvec^{n+1}\in\R^{N_u}$ 
by solving
\begin{equation}\label{scheme5eq2}
\begin{array}{l}
\displaystyle \left(
\frac{1}{k}\Mmat_u
+\Cmat_{u} +\nu\Lmat_u\right)\widetilde\Uvec^{n+1}
=\frac{1}{k}\Mmat\widetilde\Uvec^{n}-\Fvec_{u}
+\lambda \Emat_{u}\Wvec^{n+1},
\end{array}
\end{equation}
where $\displaystyle\Fvec_{u}\in \R^{N_u}$ is
defined as
\begin{equation}\label{scheme5eq2b}
\displaystyle\Fvec_{u}=
\left(\int_\Omega \nabla p_h^n\cdot\phi_j^u\right).
\end{equation}

\item Find $\Pvec^{n+1}\in\R^{N_p}$ by solving
\begin{equation}\label{scheme5eq3}
(\Lmat_{p}+\Jmat_p)\Pvec^{n+1}=-\frac{1}{k}\Fvec_{p},
\end{equation}
where $\displaystyle\Fvec_{\Pvec}\in\R^{N_p}$ is
defined as
\begin{equation}\label{scheme5eq3b}
\displaystyle\Fvec_{p}=
\left(\int_\Omega (\nabla \cdot \utilde_h^{n+1})\phi_j^p\right).
\end{equation}
\end{enumerate}

From  \eqref{scheme5eq1a}, we have
\begin{equation*}
\displaystyle \Wvec^{n+1}=
\Emat_w^{-1} \left[
\frac{1}{k}\Mmat_{d,w}\left(\Dvec^{n}-\Dvec^{n+1}\right)
-\Fvec_{w}\right],
\end{equation*}
where $\Emat_w=\lambda k\Cmat_{w} +\gamma\Mmat_{w}$ is a block-diagonal, 2-by-2 matrix, which is easy to invert by using  block Gauss-Jordan elimination.
If now we replace the above equality in equations  \eqref{scheme5eq1b} and \eqref{scheme5eq2}, and after some simple calculations, the resulting algorithm reads as:

{\sc Scheme 5: Simplified version of Scheme 4.}
\begin{enumerate}
\item Find $\Dvec^{n+1}  \in\R^{N_d}$ by solving
\begin{equation}\label{scheme6eq1}
\left(\Lmat_d+\frac{1}{k}\Mmat_{w,d}\Emat_w^{-1}\Mmat_{d,w}\right)\Dvec^{n+1}
=\Mmat_{w,d}\Emat_w^{-1}
\left[\frac{1}{k}\Mmat_{d,w}\Dvec^{n}-\Fvec_{w}\right]-\Fvec_\varepsilon,
\end{equation}
where $\displaystyle\Fvec_{w}\in \R^{N_w}$ and $ \Fvec_\varepsilon\in \mathbb{R}^{N_d}$ 
are defined in \eqref{scheme5eq1c}.
\item Find $\widetilde\Uvec^{n+1} \in \mathbb{R}^{N_u}$ by solving
\begin{equation}\label{scheme6eq2}
\begin{array}{l}
\displaystyle \left(\frac{1}{k}\Mmat_u +
\nu\Lmat_u+\Cmat_{u}\right)\widetilde\Uvec^{n+1}
=\frac{1}{k}\Mmat_u\widetilde\Uvec^{n}-\Fvec_{u}
+\lambda \Emat_u \Emat_w^{-1} \left[
\frac{1}{k}\Mmat_{d,w}\left(\Dvec^{n}-
\Dvec^{n+1}\right)-\Fvec_{w}\right],
\end{array}
\end{equation}
where $\displaystyle\Fvec_{u}\in \R^{N_u}$ is defined
in \eqref{scheme5eq2b}.

\item Find $\Pvec^{n+1}  \in \R^{N_p}$ by solving
\begin{equation}\label{scheme6eq3}
(\Lmat_{p}+\Jmat_p)\Pvec^{n+1}=-\frac{1}{k}\Fvec_{p},
\end{equation}
where $\displaystyle\Fvec_{p}\in \R^{N_p}$ 
is defined in \eqref{scheme5eq3b}.
\end{enumerate}
Observe that the matrix $\Lmat_d+\frac{1}{k}\Mmat_{w,d}\Emat_w^{-1}\Mmat_{d,w}$ is the Schur complement of system \eqref{scheme5eq1} with respect to the $\Emat_w$. 
\section{Numerical results}

In this section we present some numerical experiences that illustrate the stability, accuracy, efficiency and reliability of Scheme 5. First, we  test our numerical approximation simulating annihilation of singularities from the paper of Liu and Walkington \cite{Liu-Walkington-2000}. Next we will investigate the numerical accuracy with respect to time. In particular, we will see that the  splitting error for the director vector does not deteriorate the convergence rate of the velocity and pressure  from the non-incremental projection method for the Navier-Stokes equations \cite{Rannacher, Shen, Guermond-Miner-Shen}. Finally, we will check that the violation of the stability conditions given in $\rm (H4)$ will lead to unstable behaviors of the numerical approximations.


For all simulations, we have only used Scheme 5 for $j(p_h, q_h)=\tau (p_h-\Pi_0( p_h), \bar p_h-\Pi_0( \bar p_h))$, with $\tau=S /\nu$. The reason is that the implementation of the gradient version of $j(p_h, q_h)$ given in $\eqref{scheme3eq3b}_2$ needs to be carried out together with an extra variable to compute $\Pi_1(\nabla p^{n+1}_h)$, which requires an extra computational cost. We decided to include it in this paper because the numerical analysis is the same for both stabilization terms  in \eqref{scheme3eq3b}. This drawback can be solved by replacing the global projection operator $\Pi_1$ by a local Scott-Zhang projection operator but its stability analysis is quite different from those developed in this paper and the implementation requires some extra manipulations \cite{Badia}.     

We take the approximating spaces $\D_h, \V_h$  and $P_h$ as described in $\rm (H3)$. The numerical solutions are implemented with the help of FreeFem++ \cite{hecht}. 

\subsection{Annihilation} This numerical example is concerned with the phenomenon of annihilation of singularities. It was originally proposed in \cite{Liu-Walkington-2000} for a Dirichlet boundary condition for the director field and also performed in \cite{Becker-Feng-Prohl} for a Neumann boundary condition as considered herein. It is computed on the domain $\Omega=(-1,1)\times(-1,1)$ with the initial conditions being
\[
\u_0=\boldsymbol{0},\quad
\d_0=\frac{\tilde\d}{\sqrt{|\tilde\d|^2+0.05^2}},\mbox{ where } \tilde\d=(x^2+y^2-0.025,y),
\]
and the physical parameters being $\nu=\lambda=\gamma=1$. The discretization and penalization parameters are set as $(k, h,\varepsilon)=(0.001, 0.068986, 0.05)$. In Figures \ref{fig1} and \ref{fig2}, we show how the two singularities are carried to the origin by the velocity field forming four vortices. Snapshots of the director and velocity fields have been displayed at times t = $0.1$, $0.2 $, $0.3$ and $0.6$ in Figures \ref{fig1} and \ref{fig2}, respectively. The evolution of kinetic, elastic, and penalization energies, as well as the total energy, is depicted in Figure \ref{fig3}. Observe that the total energy decreases after each iteration as predicted by inequality (\ref{induction}). Moreover, the kinetic energy reaches its maximum level at the annihilation time. These numerical results are in good qualitative agreement with those obtained in \cite {Becker-Feng-Prohl}.  
\begin{figure}[t]
\centering
\subfigure[$\|\d\|_\infty=0.9970543.$]{
\includegraphics[scale=0.2]{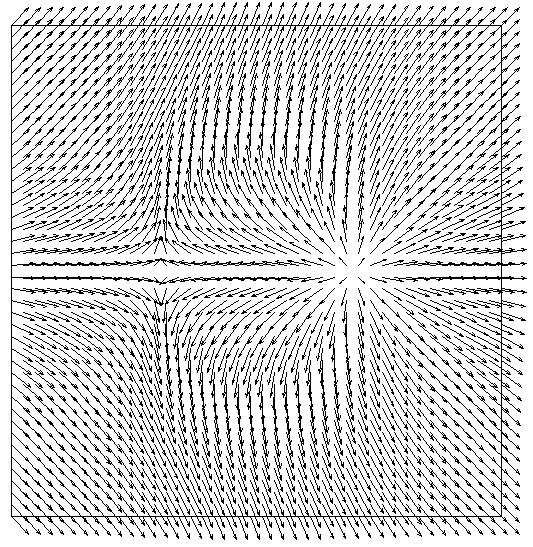}} 
\subfigure[$\|\d\|_\infty=0.9976057.$]{
\includegraphics[scale=0.2]{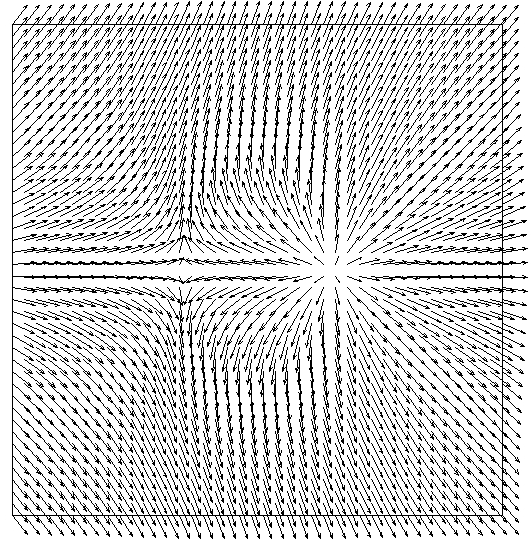}} 
\subfigure[$\|\d\|_\infty=0.9961038.$]{
\includegraphics[scale=0.2]{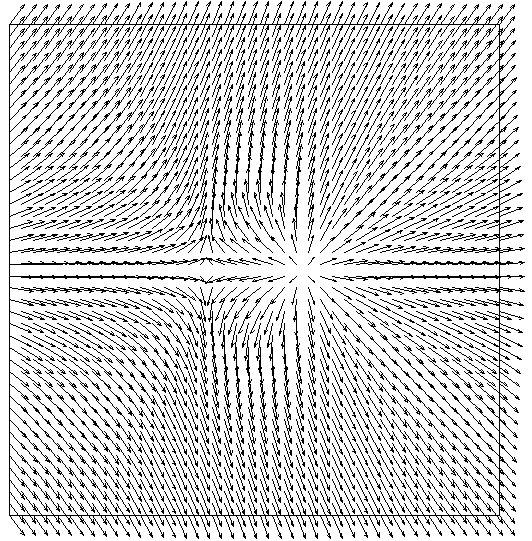}} 
\subfigure[$\|\d\|_\infty=0.9988512.$]{
\includegraphics[scale=0.2]{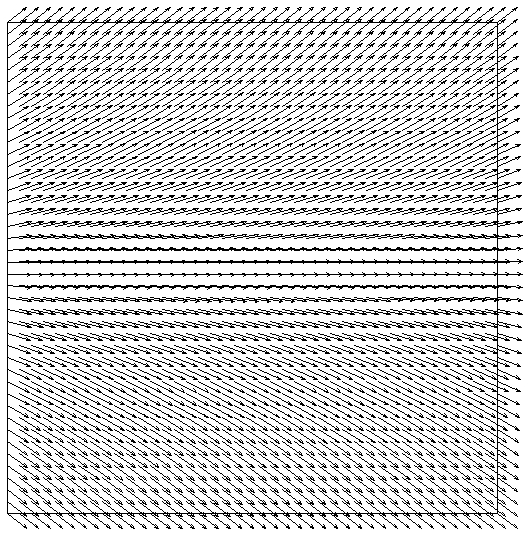}} 
\caption{Evolution of the director field at times $t=0.1, 0.2, 0,3$ and $0.6$.} 
\label{fig1}
\end{figure}

\begin{figure}[t]
\centering
\subfigure[$\|\v\|_\infty=0.2069006.$]{
\includegraphics[scale=0.21]{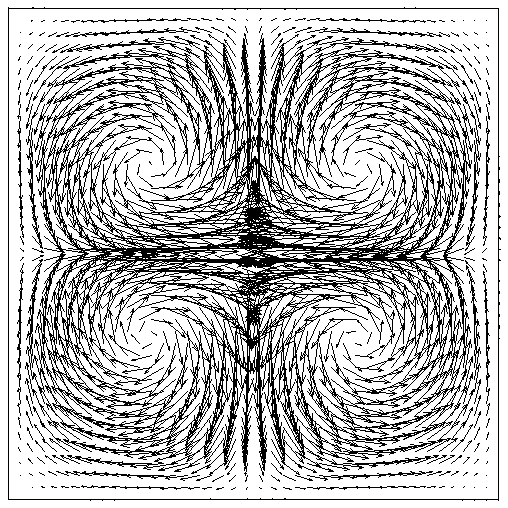}} 
\subfigure[$\|\v\|_\infty=0.1634978.$]{
\includegraphics[scale=0.21]{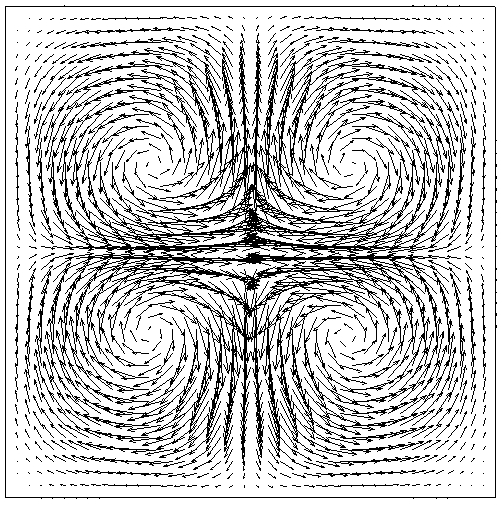}} 
\subfigure[$\|\v\|_\infty=0.2160761.$]{
\includegraphics[scale=0.21]{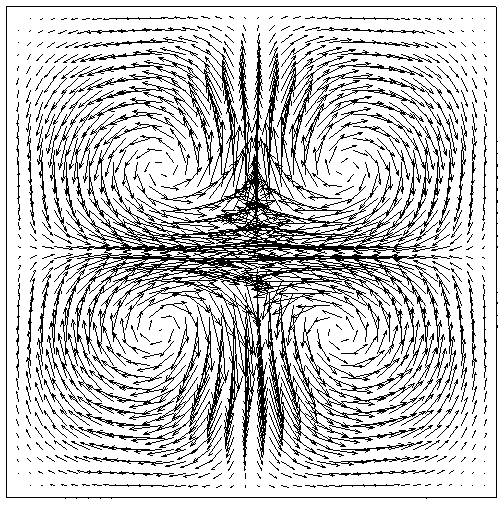}} 
\subfigure[$\|\v\|_\infty=0.001592144.$]{
\includegraphics[scale=0.21]{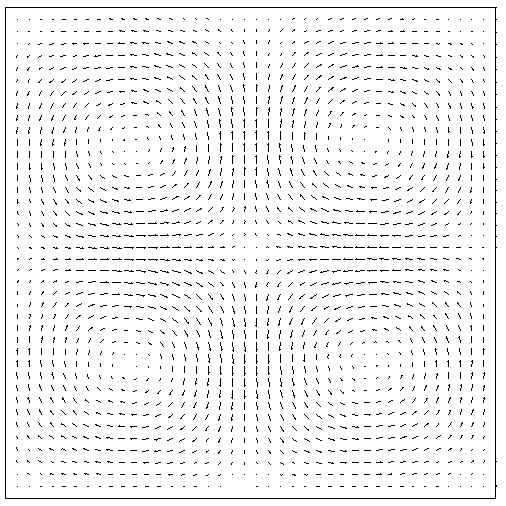}}
\caption{Evolution of the velocity field at times $t=0.1, 0.2, 0,3$ and $0.6$. 
The relative size of the vectors were modified for better visualization.
} 
\label{fig2}
\end{figure}

\begin{figure}[t]
\centering
\subfigure{
\includegraphics[scale=0.35]{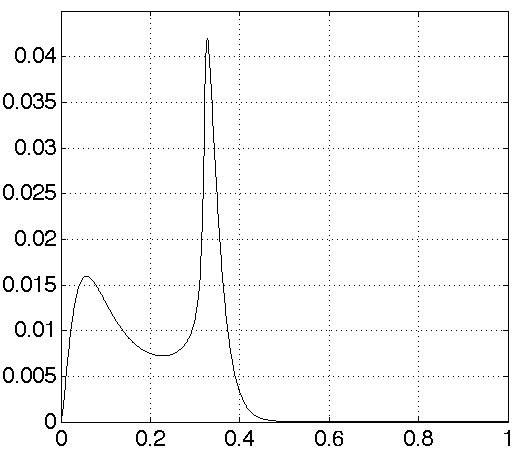}} 
\hspace{1cm}
\subfigure{
\includegraphics[scale=0.35]{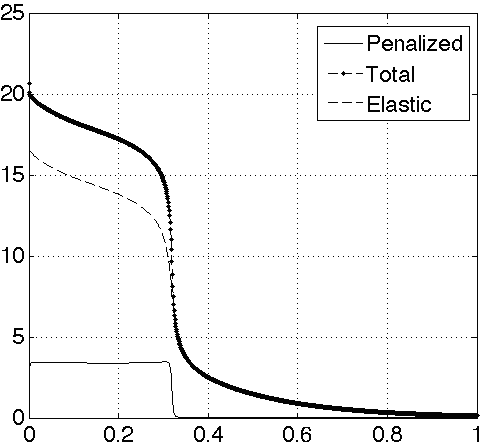}} 
\caption{Energies for the experiment of singularities. Kinetic energy 
(left) and total, elastic and penalization energy (right)} 
\label{fig3}
\end{figure}

\subsection{Convergence rate} We are now interested in the accuracy with respect to time. In doing so, we consider $\Omega=(0,1)\times(-\frac{1}{2},\frac{1}{2})$ and $\lambda=\gamma=\nu=1$. The initial data are taken as
\[
\u_0=\boldsymbol{0},\mbox{ and } \d_0=(\sin(a),\cos(a)),\quad\mbox{ where } a=\pi(\cos(\pi x)+\sin(\pi y)).
\]
which satisfies homogeneous Dirichlet conditions for the velocity field  and homogeneous Neumann  boundary conditions  for the director field. The reference solution is taken as the numerical approximation computed with the parameters $(k,h,\varepsilon)=( 1.5625\cdot 10^{-6}, 0.068986, 0.05)$. 
  
In Figure \ref{fig4} and Table \ref{Table1},  we illustrate the error bevahiour and the convergence rate on the director, velocity and pressure fields measured in the $L^2(\Omega)$- and $H^1(\Omega)$-norm versus the time step. The tests have been performed by comparing our reference solution with the numerical 
approximation computed on five time-steps $k_{i+1}=0.5 k_{i}$ for $i=1, 2, 3, 4$ with $k_1=10^{-3}$.    

The error on the velocity and director vector is of ${\mathcal O}(k)$  in the $\L^2(\Omega)$-norm, respectively, which is consistent with the results for the velocity in the context of the non-incremental projection method for the Navier-Stokes equations. The error on the director field in the $\H^1(\Omega)$-norm are of ${\mathcal O}(k)$, which means the splitting error associated to the segregation of the director field is the best that can be expected. Instead, the error on the velocity field in the $\H^1(\Omega)$-norm does not maintain the first-order accuracy from the beginning, which could mean that the theoretical order of approximation will be less than first-order. Furthermore, the error on the pressure in the $L^2(\Omega)$- and $H^1(\Omega)$-norm behaves as that on the velocity for the $H^1(\Omega)$-norm.

\begin{figure}[t]
\centering
\subfigure{
\includegraphics[scale=0.3]{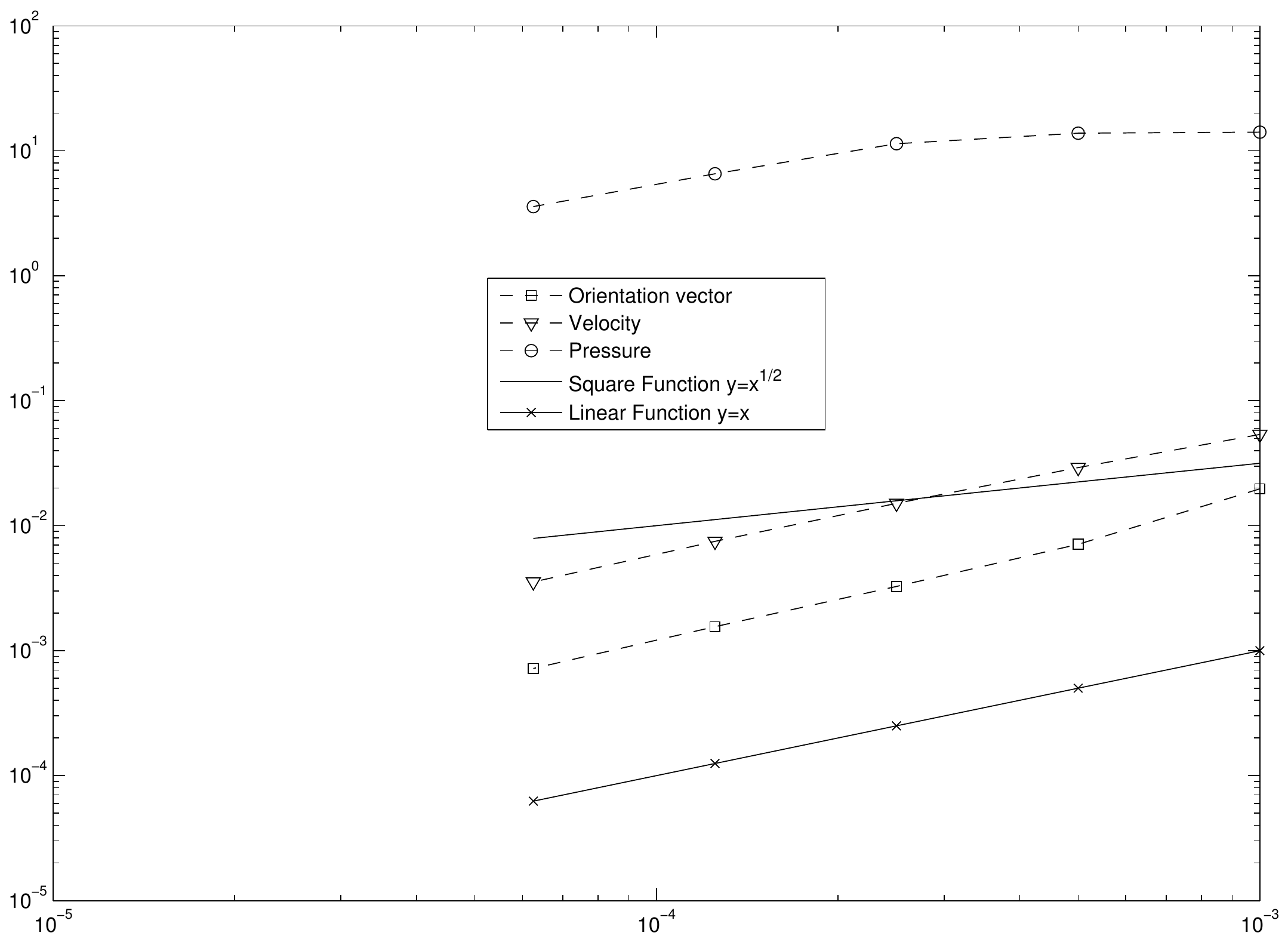}} 
\subfigure{
\includegraphics[scale=0.3]{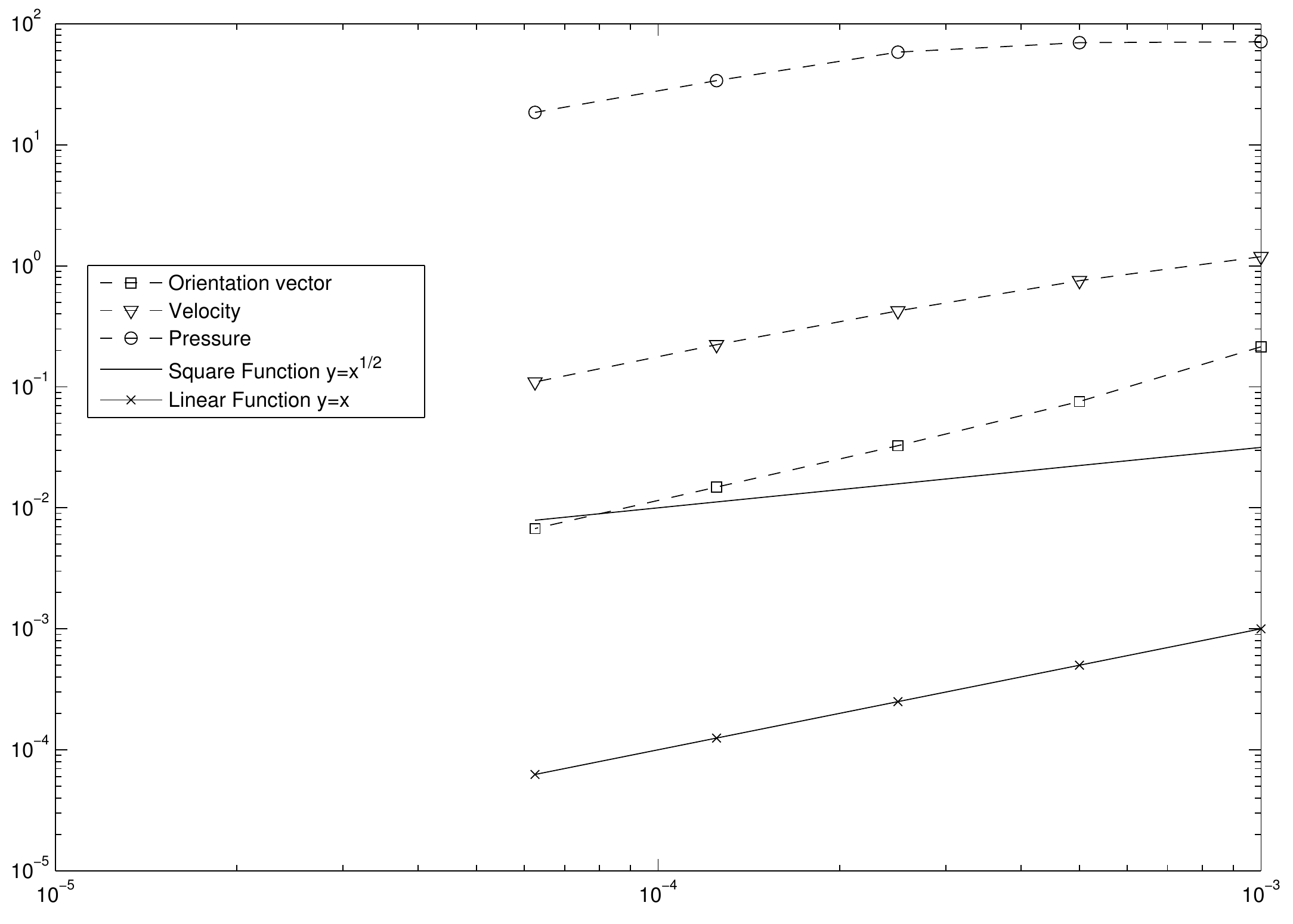}} 
\caption{ Evolution in time of the error in the $L^2(\Omega)$-norm (left) and the $H^1 (\Omega)$-norm (right)  for the director, velocity and pressure.} 
\label{fig4}
\end{figure}
\begin{table}[h]
\begin{center}
\begin{tabular}{| c | c | c | c | c | c | c  |} 
\hline 
$k$& $L^2$-rate-$p$ & $\L^2$-rate-$\v$ & $\L^2$-rate-$\d$& $H^1$-rate-$p$ & $\H^1$-rate-$\v$& $\H^1$-rate-$\d$
\\
\hline
$10^{-3}$& -- & -- & --  & --  & -- & --
\\ 
\hline
$5 \cdot 10^{-4}$& 0.0295 & 0.8801 & 1.4751 & 0.0224 &0.6581 & 1.5035
\\
\hline
$2.5 \cdot 10^{-4}$& 0.2778 & 0.9517 & 1.1205 & 0.2601 &0.8292& 1.2142
\\ 
\hline 
$1.25 \cdot 10 ^{-4}$& 0.8017 & 1.0084 & 1.0698 & 0.7832 & 0.9261& 1.1338
\\
\hline
$6.25\cdot 10^{-5}$& 0.8723 & 1.0783 & 1.1116 & 0.8742 & 1.0234 & 1.1396
\\
\hline
\end{tabular}
\end{center}
\caption{The convergence rates for the velocity, pressure and director}
\label{Table1}
\end{table}

\subsubsection{Dependence on parameters}  As it is well-known for time-splitting schemes, conditions as those given in (\rm H4) deteriorate the benefits of Scheme 5 because the number of linear systems to be solved increases when $(h, \varepsilon)$ are small. Thus one could think that these conditions reduce the efficiency of Scheme 5. However, as far as we are concerned, linear \cite{Guillen-Gutierrez} and nonlinear \cite{Becker-Feng-Prohl} Euler time-stepping algorithms require such conditions. The former in order for numerical approximations to have a priori energy estimates, while the latter in order for the associated iterative process (e.g. Newton's method) to be convergent. Therefore, we somehow have to impose conditions for $(k,h, \varepsilon)$, and hence that Scheme 5 is a good alternative for saving a lot of computational work in approximating  (\ref{Ginzburg-Landau-Problem}). 
  
To demonstrate the dependence of the numerical approximations on the discretization and penalization parameters according to the conditions given in $\rm (H4)$, we will test the sensitivity of Scheme 5 when varying the parameters $(k, h, \varepsilon)$. In doing so, we will consider the phenomenon of annihilation of singularities described above.  
In particular, we will focus on condition (\ref{Constraint2}). Thus, define $\alpha=\frac{k}{h^{3/2}\varepsilon}$ and select $\varepsilon=5\times 10^{-2}$. We want to compute our numerical approximation for $(k,h)$ where $h=0.0912396, 0.068986, 0.0463677, 0.0233754$ and $k=10^{-s}$ with $s=1, 2, 3 ,4$.  The annihilation times $T_A$ reported in Table \ref{Table2} are taken as those times where the value of the kinetic energy is maximum. In particular, we have observe that the total energy does not remain bounded as the parameter $\alpha$ becomes sufficiently large. That is, we find that our numerical approximations have no energy bounds. Therefore, Scheme 5 requires that the conditions given in $\rm (H4)$ holds in order to have a priori energy estimates in the presence of singularities.

\begin{table}[h]
\begin{center}
\begin{tabular}{|c|c|c|c|c|c|c|c|c|c}
\hline
$k\backslash h$  & $0.0912396$ & 
$0.068986$ & 
$0.0463677$& 
$0.0233754$ &\\\hline
 & 72.5697
 &110.379
 & 200.312 
 & 559.617 & $\alpha$ \\
$10^{-1}$ & \ding{55} & \ding{55}   & \ding{55} &\ding{55}  & Stab. \\
                   & $--$  & $--$   &$--$ & $--$  & $T_A$ \\
                & $--$ & $--$ &   $--$ &$--$ & $E_{kin}$ \\\hline            
 &7.25697  
 &11.0379  
 &20.0312
 &55.9617 & $\alpha$      \\
$10^{-2}$ & \ding{55}  & \ding{55}  & \ding{55}  & \ding{55}  & Stab. \\
                  & $--$ & $--$ & $--$  & $--$ & $T_A$ \\
               & $--$ & $--$ & $--$   & $--$  & $E_{kin}$ \\\hline            
  & 0.725697
  & 1.10379 
  & 2.00312    
  & 5.59617  & $\alpha$   \\
$10^{-3}$ & \ding{51} & \ding{51}    & \ding{51}  & \ding{51}  & Stab. \\
   & 0.322   &  0.328   &  0.334     &  0.338 & $T_A$ \\
   & 0.0422756  &  0.0420097    &  0.0418536     &  0.041728  &$E_{kin}$ \\\hline            
&0.0725697& 0.110379& 0.200312 &0.559617& $\alpha$   
\\
$10^{-4}$ & \ding{51} &  \ding{51}  &  \ding{51} &  \ding{51} & Stab. 
\\
& 0.3046  & 0.3105    &  0.3154     & 0.3188  & $T_A$ 
\\
&0.0490944 &  0.0487923  &  0.0485807  & 0.0484494  &$E_{kin}$ 
\\
\hline
\end{tabular}
\end{center}
\caption{Dependence of Scheme 5 on the parameters   for the annihilation phenomenon with $\varepsilon=5\times 10^{-2}$ measured by  the value of $\alpha=k/(h^{3/2}\varepsilon)$ with $h=0.0912396, 0.068986, 0.0463677, 0.0233754$ and $k=10^{-s}$ with $s=1, 2, 3 ,4$. $T_A$ is the aniquilation time}
\label{Table2}
\end{table}


\end{document}